\documentclass[a4paper,notitlepage, 8pt,reqno]{article}
  \usepackage{graphicx}
  \usepackage{mathtext}
 \usepackage[cp1251]{inputenc}
  \usepackage[T2A]{fontenc}
 \usepackage[russian,english]{babel}
 \usepackage[all,arc,poly,2cell,curve,arrow,tips]{xy}

  \usepackage{enumerate, eucal, amsthm,amsmath, amssymb}

  \allowdisplaybreaks[4]

 \theoremstyle{definition}
 \newtheorem{defn}{Definition}%[section]

 \theoremstyle{plain}
 \newtheorem{thm}{Theorem}
 \newtheorem*{thm*}{Theorem}
 \newtheorem{prop}{Proposition}
  \newtheorem*{prop*}{Proposition}
 \newtheorem{cor}{Corollary}
  \newtheorem*{cor*}{Corollary}
 \newtheorem{lem}{Lemma}
  \newtheorem*{lem*}{Lemma}

 \theoremstyle{remark}

 \newtheorem*{remark*}{Remark}

 \renewcommand{\abstractname}{}

  \newcounter{ab}
\title{ A problem of restriction  $g_n\downarrow g_{n-1}$  for  Lie algebras of series $A,B,C,D$. }

 \author{D. V.  Artamonov}
  \date{}

\begin{document}
 \maketitle

\renewcommand{\abstractname}{}

\begin{abstract}
Using the Zhelobenko's approach we investigate a branching of an irreducible representation of  $g_n$ under the restriction of algebras $g_n\downarrow g_{n-1}$, where  $g_n$ is a Lie algebra of type
$B_n$, $C_n$, $D_n$ or a Lie algebra of type $A$, where in this case we put
$g_{n}=\mathfrak{gl}_{n+1}$, $g_{n-1}=\mathfrak{gl}_{n-1}$. We give a new  explicit description of  the space of the  $g_{n-1}$-highest vectors,  then we construct  a base in this space.
The case  $n=2$ is considered separately for different algebras, but a passage from $n=2$ to an arbitrary $n$ is the same for all series  $A$,
$B$, $C$, $D$.  This new procedure has the following advantage: it establishes  a relation between spaces of  $g_{n-1}$-highest vectors for different series of algebras.  This procedure  describes an extension of Gelfand-Tsetlin tableaux to the left.

\end{abstract}

%\tableofcontents

\section{Introduction}

In the book  \cite{zh2} Zhelobenko used realizations of a
representation of a simple Lie algebra $\mathfrak{g}$ in the space
of all functions on a corresponding Lie group and in the space of
functions on a subgroup of unipotent upper-triangular matrices. In
\cite{zh2} conditions that define a representation of a given highest
weight in these realization are presented. These realization are very convenient
for an investigation of restriction problems. A problem of a
restriction $\mathfrak{g}\downarrow \mathfrak{k}$, where
$\mathfrak{g}$ is a Lie algebra and $\mathfrak{k}$ is it's
subalgebra is a problem of a description   of $\mathfrak{k}$-highest
vectors in an irreducible representation of $\mathfrak{g}$. A
solution of such a problem is a key step in a construction of a
Gelfand-Tsetlin type base in a representation of a Lie algebra.

 In \cite{zh2} the cases of restriction problems
  $\mathfrak{gl}_n\downarrow \mathfrak{gl}_{n-1}$ and  $\mathfrak{sp}_{2n}\downarrow \mathfrak{sp}_{2n-2}$ are considered. In the first case a construction of a base in the space of  $\mathfrak{gl}_{n-1}$-highest vectors encoded by Gelfand-Tsetlin tableaux is given. Then using a realization in the space of functions on the
  subgroup of unipotent upper-triangular matrices in \cite{zh2} it is shown that these restriction problems are equivalent.
  Using this equivalence the problem $\mathfrak{sp}_{2n}\downarrow \mathfrak{sp}_{2n-2}$ is solved.

Later the problems  $g_n\downarrow g_{n-1}$, where
$g_n=\mathfrak{o}_{2n+1}$, $\mathfrak{sp}_{2n}$
  or $\mathfrak{o}_{2n}$ in the realization in the space of functions on unipotent upper-triangular matrices
  was investigated by V.V. Shtepin in \cite{sh1},  \cite{sh2}, \cite{sh3}.
  He obtained  solutions of these problems but a relation with the problem $\mathfrak{gl}_{n+1}\downarrow\mathfrak{gl}_{n-1}$  was not discovered.

A.I. Molev in \cite{Mol1}, \cite{Mol2}, \cite{Mol3} (see also
\cite{M}) obtained a solution of a problem of construction of a
Gelfand-Tsetlin type base for a finite dimensional representation of
a  Lie algebra $g_n=\mathfrak{o}_{2n+1}$, $\mathfrak{sp}_{2n}$
  or $\mathfrak{o}_{2n}$.  Molev gave a construction of base vectors and obtained formulas for the action of generators of the algebra in this base.
  But he used another technique. To obtain  a solution of the problem  $g_n\downarrow g_{n-1}$  an action of a Yangian
  on the space of  $g_{n-1}$-highest  vectors with  a fixed highest
  weight was constructed. For all series of algebras Gelfand-Tsetlin type tableaux constructed by Molev
  have the following property.  There right part has a structure
  that depends on a series of the algebra. But as $n$ increases a
  tableau grows to the left and the structure of the extension of the tableau
  does not depend on the type of the algebra.

This  fact is a starting point of the present paper. The main result is a new construction of a solution of the problem $g_n\downarrow g_{n-1}$  (Theorem \ref{osnl1}) establishing a relation between solutions for different series  (Corollary \ref{osncor}).
We investigate
the problem $g_n\downarrow g_{n-1}$  in the  realization in the space of
function on the whole group.  %and on the subgroup of upper triangular matrices as in \cite{zh2}.

 First of all for series  $B$, $C$, $D$ we give explicit conditions that  define an irreducible representation
  of a given highest weight in the realization in function on the whole group (see Theorem \ref{tzh}, where the conditions are formulated, and Section \ref{islm}, where these conditions are rewritten explicitely),  in \cite{zh2} it is done only for the series $A$.
Then we obtain a description of functions on the group that
correspond to $g_{n-1}$-highest vectors
  (Theorem \ref{le5} and Theorem \ref{lempoc}).

   Finally in the main Sections  \ref{abcd2}, \ref{abcdn}  we give a procedure of a construction of a Gelfand-Tsetlin type base in this space.
   Firstly  separately for  the series
     $A$, $B$, $C$,  $D$ it is done for  $n=2$ (see Section \ref{abcd2}),  and then for all these series
     simultaneously we describe a passage from the problem $g_2\downarrow g_1$ to $g_n\downarrow
g_{n-1}$ (Section \ref{abcdn}, Theorem \ref{osnl1} ).  This
procedure is interpreted as an extension of a Gelfand-Tsetlin
tableau to the left. We call this procedure an extension of  a problem
of restriction $g_n\downarrow g_{n-1}\supset g_2\downarrow g_1$.

\section{Introduction}

\subsection{Algebras $\mathfrak{o}_{2n+1}$,  $\mathfrak{o}_{2n}$, $\mathfrak{sp}_{2n}$}
\label{algl}

The Lie algebras $\mathfrak{o}_{2n}$, $\mathfrak{sp}_{2n}$ are
considered as subalgebras in the algebra of all matrices $2n\times
2n$, whose  rows are indexed  $i,j=-n,...,-1,1,...,n$, the algebra
$\mathfrak{o}_{2n+1}$ is a subalgebra in the space of all
$(2n+1)\times (2n+1)$-matrices, whose rows and columns are indexed
by $i,j=-n,...,-1,0,1,...,n$.

The algebras $\mathfrak{o}_{2n+1}$ and   $\mathfrak{o}_{2n}$
 are generated by

\begin{equation}\label{fbd}F_{i,j}=E_{i,j}-E_{-j,-i},\end{equation}

where $i,j=-n,...,-1,0,1,...,n$ in the case  $\mathfrak{o}_{2n+1}$
and $i,j=-n,...,-1,1,...,n$  in the case $\mathfrak{o}_{2n}$.

The algebra $\mathfrak{sp}_{2n} $ is generated by
\begin{equation}\label{fd}F_{i,j}=E_{i,j}-sign(i)sign(j)E_{-j,-i},\end{equation}
where $i,j=-n,...,-1,1,...,n$.

Subalgebras  $\mathfrak{o}_{2n-1}$,   $\mathfrak{o}_{2n-2}$,
$\mathfrak{sp}_{2n-2} $ are generated by  $F_{i,j}$ for
$i,j=-n,...,-2,2,...,n$  in the case $C$, $D$ and
$i,j=-n,...,-2,0,2,...,n$ in the case $B$.

 Also we use the algebra $\mathfrak{gl}_{n+1}$ of all matrices $(n+1)\times (n+1)$, whose rows and columns
 are indexed by  $i,j=-n,...,-1,1$.
 This algebra is generated by  $E_{i,j}$,   $i,j=-n,...,-1,1$, choose a subalgebra   $\mathfrak{gl}_{n-1}$, spanned by  $E_{i,j}$,    $i,j=-n,...,-2$.

\subsection{Functions on the group}
\label{fugu} We use a realization of a representation of a Lie
algebra in the space of functions on the  Lie group $G=O_{2n+1}$,
$Sp_{2n}$, $O_{2n}$, $GL_{n+1}$ (see \cite{zh2}). Onto a functions
$f(g)$, $g\in G$ an elements $X\in G_{}$ acts  by right shifts according to  the formula

\begin{equation}
\label{xf}
(Xf)(g)=f(gX).
\end{equation}

Fix a highest weight $[m]=[m_{-n},...,m_{-1}]$ in the cases
$O_{2n+1}$, $Sp_{2n}$, $O_{2n}$ and
$[m]=[m_{-n},...,m_{-1},m_{1}=0]$ in the case $GL_{n+1}$. In the
cases $G=Sp_{2n}$,  $GL_{n+1}$ the highest weight is integer and
in the case   $ O_{2n+1}$, $O_{2n}$ it can be half-integer. In the cases   $GL_{n+1}$,  $O_{2n+1}$, $Sp_{2n}$ the weight is non-negative and in the case  $O_{2n}$ we can have $m_{-1}<0$.

\subsection{Determinants and  a formula for the highest vector}
\label{stvv}

 Let $a_{i}^{j}$, $i,j=-n,...,0,...,n$
   be a function of a matrix element on the group $GL_{2n+1}$. Here $j$ is a row index and  $i$ is a column index.
   Later we consider restrictions of these functions onto  $G\subset GL_{2n+1}$.
    Thus relations between $a_{i}^{j}$ appear.

    %, но между  дифференциальными операторами мы соотношений вводить не будем.

%Тогда для серии  $A$

%\begin{equation}
%\label{eij}
%E_{i,j}=\sum_{k}
%a_{i}^{k}\frac{\partial }{\partial  a_{j}^{k}},
%\end{equation}
%для серий же  $B$, $C$, $D$ строятся из этих операторов по формулам из раздела \ref{algl}

Put

\begin{equation}
\label{dete}
a_{i_1,...,i_k}:=det(a_i^j)_{i=i_1,...,i_k}^{j=-n,...,-n+k-1}.
\end{equation}

That is we take a determinant of a submatrix of a matrix $(a_i^j)$,
formed by rows $-n,...,-n+k-1$ and columns $i_1,...,i_k$. By
formulas \eqref{xf}, \eqref{dete}  we conclude that an operator
$E_{i,j}$ acts onto a determinant by changing column indices
according to the ruler

\begin{equation}
\label{edet1}
E_{i,j}a_{i_1,...,i_k}=a_{\{i_1,...,i_k\}\mid_{j\mapsto i}},
\end{equation}

where $.\mid_{j\mapsto i}$ is an operation of a substitution of $i$
instead of $j$, if  $j\notin\{i_1,...,i_k\}$ then we obtain $0$. An
operator $F_{i,j}$ for the series $B$, $C$, $D$ acts by formulas
\eqref{fbd}, \eqref{fd}.

In the case of the series $D$ put

\begin{equation}
\bar{a}_{i_1,...,i_n}:=det(a_i^j)_{i=i_1,...,i_n}^{j=-n,...,-2,1}.
\end{equation}

Then

\begin{equation}
(a_{-n,...,-2,-1})^{-1}= \bar{a}_{-n,...,-2,1}.
\end{equation}

This equality in the case  $n=2$ is checked by direct computation,
the case of an arbitrary  $n$  is considered as in Lemma \ref{ssp4}.

Now let us give a formula for the highest vector (see \cite{zh2}).
The vector
\begin{align}\begin{split}
\label{stv} &v_0=\prod_{k=-n}^{-2} (a_{-n,...,-k})^{m_{-k}-m_{-k+1}}
a_{-n,...,-2,-1}^{m_{-1}} \text{ in the cases $A$, $B$, $C$, $D$ and
$m_{-1}\geq 0$},\\
&v_0=\prod_{k=-n}^{-2} (a_{-n,...,-k})^{m_{-k}-m_{-k+1}}
\bar{a}_{-n,...,-2,1}^{-m_{-1}} \text{ in the case $D$ and $m_{-1}<
0$},
\end{split}\end{align}
%где  $1$ есть тождественно равная единице функция,
is a highest vector for $g_{n}$ with the weight
 $[m_{-n},...,m_{-1}]$ for the series  $B$, $C$, $D$  and  $[m_{-n},...,m_{-1},0]$ for the series  $A$.
 Note that in the case of an integer highest weight this is a polynomial function. And in the case of
 $B$, $D$
 and a half-integer highest weight this is not a polynomial function
 since it contains a factor
$a_{-n,...,-2,-1}^{m_{-1}}$ or  $\bar{a}_{-n,...,-2,-1}^{-m_{-1}}$.

\subsection{Functions on the subgroup $Z$}
\subsubsection{The action on the functions on $Z$}

 For a typical matrix $X\in G$  one has a Gauss decomposition \begin{equation}\label{gaus}X=\zeta\delta z,\end{equation}
into a product of a lower-triangular unipotent matrix, a diagonal
matrix and an upper-triangular unipotent matrix. We denote a
subgroup of upper-triangular unipotent matrices  as $Z$.
For an element $g\in G$ and a matrix $X$ one has the Gauss decomposition

 $$Xg=\tilde{\zeta}\tilde{\delta}\tilde{ z}.$$

Put
$\tilde{\delta}=diag(\tilde{\delta}_{-n},\tilde{\delta}_{-n+1},...)$.
The action of $G$ on the functions on $Z$ is given by  the formula

 $$
(gf)(z)=\tilde{\delta}_{-n}^{m_{-n}}\tilde{\delta}_{-n+1}^{m_{-n+1}}...\tilde{\delta}_{-1}^{m_{-1}}f(\tilde{z})
 $$

\subsubsection{Conditions that define an irreducible representations in the functions on  $Z$}

\label{funaz}

Functions on $Z$ that form an irreducible representation  with a given
highest weight are selected by the following condition (see
\cite{zh2}).

\begin{enumerate}
%\item  A function is a polynomial in matrix elements\footnote{Even in the case of half-integer highest weight}.
\item A function satisfies the indicator system.
% $L_{\alpha_{-i}}^{r_{-i}+1}f=0$,  $i=1,...,n$,   where $L_{\alpha_{-i}}$
% is a left infinitesimal shift by an  element  corresponding
% to a positive simple root $\alpha_{-i}$, and $r_{-i}$ are written explicitly below % $r_{-i}=<\alpha_{-i},[m]>$, where  $[m]$ is a highest weight.
\end{enumerate}

%Let us  write explicitly equations of the indicator system.
 The
indicator system  is a system of PDE of type
\begin{align*}
&L_{-n,-n+1}^{r_{-n}+1}f=0,...,L_{-2,-1}^{r_{-2}+1}f=0,\,\,\,,L_{-1,1}^{r_{-1}+1}f=0\text{ in the cases  $A$, $C$},\\
&L_{-n,-n+1}^{r_{-n}+1}f=0,...,L_{-2,-1}^{r_{-2}+1}f=0,\,\,\,,L_{-1,0}^{r_{-1}+1}f=0\text{ in the case  $B$,}\\
&L_{-n,-n+1}^{r_{-n}+1}f=0,...,L_{-2,-1}^{r_{-2}+1}f=0,\,\,\,L_{-2,1}^{r_{-1}+1}f=0\text{
in the case  $D$}
\end{align*}

  Here  $L_{-i,-j}$ is an operator acting on a function  $f(z)$
and doing a left infinitesimal shift of a function  $f(z)$ onto
$F_{-i,-j}$ for the  series $B$, $C$, $D$ and a shift by $E_{-i,-j}$
for the series $A$.
%Показатель   $r_{-i}$
% определяются равенством $<\alpha_i,[m]>$, где $\alpha_i$ корень, соответствующий элементу $F_{-i,-j}$ или $E_{-i,-j}$,

An exponent  $r_{-i}$ is written  as follows

\begin{align}
\begin{split}
\label{rb}
&r_{-n}=m_{-n}-m_{-n+1},...,r_{-2}=m_{-2}-m_{-1},\,\, r_{-1}=m_{-1}\text{ in the cases  $A$, $C$,}\\
&r_{-n}=m_{-n}-m_{-n+1},...,r_{-2}=m_{-2}-m_{-1},\,\, r_{-1}=2m_{-1}\text{ in the case  $B$,}\\
&r_{-n}=m_{-n}-m_{-n+1},...,r_{-2}=m_{-2}-|m_{-1}|,\,\,
r_{-1}=m_{-2}+|m_{-1}|\text{ in the case  $D$}.
\end{split}
\end{align}

It turns out that such a function  is a polynomial in matrix elements\footnote{Even in the case of half-integer highest weight}.

In \cite{zh2} and also in \cite{sh3}   in the case of the series $D$  the exponents $r_{-2}=m_{-2}-m_{-1}$ and   $r_{-1}=m_{-2}+m_{-1}$ are used.
The reason is that in  \cite{zh2}, \cite{sh3} and in the present paper a different choice of the highest vector
  in the realization in functions on  $G$ is done (see \eqref{stv}).   And the highest vector must satisfy the indicator system.

\subsection{ Conditions that define an irreducible representations in the functions on $G$}

\subsubsection{The general Theorem}

In \cite{zh2} in the case $GL_{n+1}$  the following statement is
proved.  In the realization in the space of functions on the group
$G$ an irreducible representation with the highest weight
$[m_{-n},...,m_{-1},0]$  and a highest vector \eqref{stv} are
selected by conditions

\begin{enumerate}
%\item $f$ is a polynomial in matrix elements.
\item $L_{-}f=0$, where  $L_{-}$ is a left infinitesimal shift by an arbitrary element
   of  $GL_{n+1}$, corresponding to a negative root.
\item $L_{-i,-i}f=m_{-i}f$, where  $L_{-i,-i}$, $i=1,...,n$ is a left infinitesimal shift by an
element of  $GL_{n+1}$, corresponding to a Cartan
element$E_{-i,-i}$.
\item $f$ satisfies the indicator system.% $L_{\alpha_{-i}}^{r_{-i}+1}f=0$,  $i=1,...,n$.
\end{enumerate}

Let us prove an analogous statement for the series $B$, $C$, $D$.

\begin{thm}
\label{tzh}
In the realization in functions on the whole group for the series  $B$, $C$, $D$ an irreducible representation with the highest vector given in Section  \ref{stv},
is selected by  the conditions $1,3$, by the condition $2$ where we change $E_{-i,-i}$ to
$F_{-i,-i}$ in the definition of $L_{-i,-i}$.
%, and by the modified condition $1$:

%\begin{itemize}
%  \item  If the highest weight is integer and non-negative then  $f$ is a polynomial in determinants.
%  \item  In the case of the series  $D$  if the highest weight is integer and  $m_{-1}<0$, then  $f$  is a linear combination of products of determinants. In each product a determinant occurs in a %non-negative integer power except   $a_{-n,...,-2,-1}$, which can occur powers  $0,-1,-2,...,m_{-1}$.
  %\item If the highest weight is half-integer and non-negative then $f$  is a linear combination of products of determinants. In each product a determinant occurs in a non-negative integer power except  $a_{-n,...,-2,-1}$, which can occur powers   $m_{-1},m_{-1}-1,...,\frac{1}{2},-\frac{1}{2}$.
  %\item  In the case of the series  $D$  if the highest weight is half-integer and  $m_{-1}<0$, then $f$  is a linear combination of products of determinants. In each product a determinant occurs in a %non-negative integer power except  $a_{-n,...,-2,-1}$, which can occur powers  $\frac{1}{2},-\frac{1}{2},....,m_{-1}$.
%\end{itemize}
\end{thm}

\proof

The scheme of the proof is the following. Firstly we derive  formulas for the action of a left infinitisimal shift. Using them we prove the the highest vector satisfies the conditions 1-3. The main difficulty is to prove that the highest vector satisfies the indicator system.  Then we easily prove that an arbitrary vector of the representation satisfies the conditions 1-3. Secondly we prove that among functions that satisfy conditions 1-3 there is nothing but functions that form a representation with the highest vector  \eqref{stv}.

Take a determinant
$$
a_{i_1,...,i_k}=det(a_i^j)_{i=i_1,...,i_k}^{j=-n,...,-n+k-1}
$$

introduce a notation

$$
a^{-n,...,-n+k-1}_{i_1,...,i_k},
$$
 where upper indices are row indices. Then the operator  $L_{-i,-j}$ of the left infinitesimal shift acts on the upper indices  $-n,...,-n+k-1$ by the following. For the series   $A$  the left infinitesimal shift  by  $E_{-i,-j}$  act as follows

\begin{equation}
\label{lij}
L_{-i,-j}a^{-n,...,-n+k-1}_{i_1,...,i_k}=a^{\{-n,...,-n+k-1\}\mid_{-i\mapsto -j}}_{i_1,...,i_k},
\end{equation}

and the left infinitesimal shift  by   $F_{-i,-j}$ for the series $B$,
$C$, $D$ is expressed though these operators.

Analogously

\begin{equation}
\label{fii}
L_{-i,-i}a^{-n,...,-n+k-1}_{i_1,...,i_k}=
a^{\{-n,...,-n+k-1\}}_{i_1,...,i_k}\text{ if }
-i\in\{-n,...,-n+k-1\},
 \text{ $0$ otherwise }.
\end{equation}

On a product of determinants  $L_{-i,-i}, L_{-i,-j}$ act according to the Leibnitz
ruler.

From these formulas we see that the conditions  1-3 for an irreducible representations with the highest vector defined in Section  \ref{stvv} do hold.

First we prove that conditions 1-3  hold for the highest vector. The
fact that conditions 1 and 2 hold is proved by very easy direct
computations. We need to prove that the condition 3 holds.

 From the formula  \eqref{lij} it follows that  operators $L_{-i,-i+1}$ for $i=-n,...,-3$ for all  series  act onto
    \eqref{stv}  by acting only onto
    $a_{-n,...,-i}^{m_{-i}-m_{-i+1}}$. Also
   $L_{-2,-1}$ for the series $A$, $B$, $C$, and $D$ in the case  $m_{-1}\geq 0$,  $L_{-2,1}$
    for the series  $D$ in the case  $m_{-1}<0$  act onto    \eqref{stv}  by acting only onto $a_{-n,...,-2}^{m_{-2}-m_{-1}}$.
    The operator   $L_{-1,1}$ for the series  $A$, $C$,  $L_{-1,0}$ for the series  $B$  acts on to a determinant of order $n$.

    But there are also special operators.
    For the series  $D$   and  $m_{-1}\geq 0$  the operator $L_{-2,1}$ acts onto determinants of orders $n-1$ and  $n$.
    And in the case  $m_{-1}< 0$ so does the operator  $L_{-2,-1}$.

  Since the power $r_{-i}=m_{-i}-m_{-i+1}$ of the determinant  $a_{-n,...,-n+i-1}$ in  \eqref{stv}  is an integer then due to  \eqref{lij} the equations   $L_{-i,-i+1}^{r_{-i}+1}v_0=0$, $i=-n,...,-2$   hold.  For the same reason the equation $L_{-1,1}^{r_{-1}+1}v_0=0$ for the series $A$, $C$ hold.

 Let us check that $L_{-1,0}^{m_{-1}+1}v_0=0$ for the series  $B$. For the series  $B$ one has

 \begin{align*}
 &L_{-1,0}a_{-n,...,-2,-1}^{-n,...,-2,-1}=a_{-n,...,-2,-1}^{-n,...,-2,0},\,\,\,L_{-1,0}^2a_{-n,...,-2,-1}^{-n,...,-2,-1}=-a_{-n,...,-2,-1}^{-n,...,-2,1},\\&L_{-1,0}^3a_{-n,...,-2,-1}^{-n,...,-2,-1}=0
 \end{align*}

From these formulas we obtain that $L_{-1,0}^{2m_{-1}+1}v_0=0$ for the series  $B$ and an integer highest weight.

One also has

 \begin{align*}
 &L_{-1,0}(a_{-n,...,-2,-1}^{-n,...,-2,-1})^{1/2}=\frac{1}{2}a_{-n,...,-2,-1}^{-n,...,-2,0}(a_{-n,...,-2,-1}^{-n,...,-2,-1})^{-1/2},\\
 &L_{-1,0}^2a_{-n,...,-2,-1}^{-n,...,-2,-1}=-\frac{1}{2}a_{-n,...,-2,-1}^{-n,...,-2,1}(a_{-n,...,-2,-1}^{-n,...,-2,-1})^{-1/2}-\\&-
 \frac{1}{2}(a_{-n,...,-2,-1}^{-n,...,-2,0})^2(a_{-n,...,-2,-1}^{-n,...,-2,-1})^{-3/2}=0
 \end{align*}

To prove these formulas one must use the relation $$a_{-n,...,-2,-1}^{-n,...,-2,1}a_{-n,...,-2,-1}^{-n,...,-2,-1}=-\frac{1}{2}(a_{-n,...,-2,-1}^{-n,...,-2,0})^2.$$ This relation can be derived by analogy with Lemma \ref{ssp4} below.  From these formulas we conclude that  $L_{-1,0}^{2m_{-1}+1}v_0=0$ for the series  $B$ and a  half-integer highest weight.

Now consider the case of the series  $D$ and  $m_{-1}\geq 0$.  Let us check that $L_{-2,1}^{m_{-2}+m_{-1}+1}v_0=0$.   One has

\begin{align*}
&L_{-2,1}a_{-n,...,-2}^{-n,...,-2}=a_{-n,...,-2}^{-n,...,1},\,\,\, L_{-2,1}a_{-n,...,-2}^{-n,...,1}=0,\\
&L_{-2,1}a_{-n,...,-2,-1}^{-n,...,-2,-1}=2a_{-n,...,-2,-1}^{-n,...,1,-1},\,\,\,
L_{-2,1}a_{-n,...,-2,-1}^{-n,...,1,-1}=a_{-n,...,-2,-1}^{-n,...,1,2},\,\,\,
L_{-2,1}a_{-n,...,-2,-1}^{-n,...,1,2}=0.
\end{align*}

Thus $v_0$ vanishes under the action of   $L_{-2,1}$ to the power
which equals one plus the power of  $a_{-n,...,-2}$  and the double
power of   $a_{-n,...,-2,-1}$,  that is $v_0$ vanishes under the
action of  $L_{-2,1}$ to the power
$1+(m_{-2}-m_{-1})+2m_{-1}=1+m_{-2}+m_{-1}$.

For the series  $D$ and  $m_{-1}< 0$.  Let us check that $L_{-2,-1}^{m_{-2}+m_{-1}+1}v_0=0$. Оne has

\begin{align*}
&L_{-2,-1}a_{-n,...,-2}^{-n,...,-2}=a_{-n,...,-2}^{-n,...,-1},\,\,\, L_{-2,-1}a_{-n,...,-2}^{-n,...,-1}=0,\\
&L_{-2,-1}a_{-n,...,-2,1}^{-n,...,-2,1}=2a_{-n,...,-2,-1}^{-n,...,-1,1},\,\,\,
L_{-2,-1}a_{-n,...,-2,-1}^{-n,...,-1,1}=-a_{-n,...,-2,-1}^{-n,...,-1,2},\,\,\,
L_{-2,1}a_{-n,...,-2,-1}^{-n,...,-1,2}=0.
\end{align*}

Thus  $v_0$ vanishes under the action of   $L_{-2,-1}$ to the power
$1+(m_{-2}-m_{-1})+2m_{-1}=1+m_{-2}+m_{-1}$.

Thus the highest vector satisfies the indicator system.

The fact that the conditions  1-3 hold for an arbitrary vector of the representation follows form the fact that  left and right shifts commute. And an arbitrary vector of the representation is a linear combination of right shifts of the highest vector.

Thus we need to check that among functions that satisfy the conditions 1-3  there are no other functions.

Let us use the following statement form \cite{zh2}.

\begin{prop}\label{claizh} If one restricts form  $G$ to  $Z$ functions that form an irreducible representation with the highest vector  \eqref{stv} one obtains a bijection with the space of functions on  $Z$ that from an irreducible representation with the same highest weight. These function on   $Z$  are selected by conditions from Section  \ref{funaz}.
\end{prop}

So let us be given a function on  $G$, that satisfy the conditions 1-3.  It's restriction on $Z$ satisfies the condition 1 from the Section \ref{funaz}. Thus this restriction belongs to the irreducible representation in the realization in functions on $Z$.  Using the statement above we conclude that the initial function on
   $G$ belongs to an irreducible representation with a highest vector \eqref{stv}.

\endproof
\subsubsection{Solutions of the indicator system and the equations $L_{-i,-i}f=m_{-i}f$.}

\label{islm}
%Возьмём функцию от определителей

In the proof of Theorem  \ref{tzh} the formulas for the action of  $L_{-i,-i}$ were derived. They give us the following statement.

 \begin{lem}
 \label{lfmf}
  Solutions of  the system of equations $L_{-i,-i}f=m_{-i}f$ that are functions of determinants are described as follows. If one expands the function $f$ into a sum
   (maybe infinite) of products of determinants then in each summand the
  sum of powers of determinants of order  $n-i+1$ equals
   $r_{-i}$ for $i=n,...,2$.  Also a sum  of powers of determinants of order  $n$ equals $m_{-1}$ in the case $m_{-1}\geq 0$ and
   $-m_{-1}$ in the case  $m_{-1}<0$.

\end{lem}

Let us find conditions that select functions that satisfy the indicator system.

\begin{lem}
\label{l1} If the case of an integer
 non-negative highest weight the solution of the indicator system are polynomials that satisfy the condition of Lemma  \ref{lfmf}.
\end{lem}

\begin{proof}
Since the highest vector   \eqref{stv} is a polynomial in determinants and the
 space of such function is invariant under the  right action of algebra, then an
  arbitrary vector can be represented as a polynomial in determinants. Then the statement of Lemma follows immediately from  \eqref{lij}.
\end{proof}

\begin{lem}
\label{l2} If the highest weight is half-integer then among functions that satisfy
\ref{lfmf} solutions of the indicator system are functions of type

\begin{align}
\begin{split}
\label{f12}
&f=(a_{-n,...,-2,-1})^{\frac{1}{2}}f_1+a_{-n,...,-2,0}(a_{-n,...,-2,-1})^{-\frac{1}{2}}f_2\text{ for the series  $B$}\\
&f=(a_{-n,...,-2,-1})^{\frac{1}{2}}f_1+a_{-n,...,-1,1}(a_{-n,...,-2,-1})^{-\frac{1}{2}}f_2\text{
for the series  $D$ and  $m_{-1}\geq 0$},\\
&f=(\bar{a}_{-n,...,-2,1})^{\frac{1}{2}}f_1+\bar{a}_{-n,...,-1,1}(\bar{a}_{-n,...,-2,1})^{-\frac{1}{2}}f_2\text{
for the series  $D$ and $m_{-1}< 0$}.
\end{split}
\end{align}
where  $f_1$ and   $f_2$ are polynomials in determinants.

\end{lem}

\begin{proof}
Consider the case of the series  $B$.

 Let us first prove that a solution of the indicator system looks as \eqref{f12}.
 A function of determinants satisfies automatically the condition 1 of Theorem \ref{tzh}, so our function
  is a vector of a representation with the highest  vector \eqref{stv}.

Consider first the case of the highest vector   $(a_{-n,...,-2,-1})^{\frac{1}{2}}$.  An arbitrary element of the representation is represented as a linear combination of

$$
F_{0,-1}^{p_{-1}}F_{-1,-2}^{p_{-2}}...F_{-n+1,-n}^{p_{-n}}(a_{-n,...,-2,-1})^{\frac{1}{2}},
$$

 this vector is non-zero only in the case
$p_{-n}=...=p_{-2}=0$,  and $p_{-1}=0$ or  $1$. Indeed for
$p_{-1}=1$ we get the vector
$\frac{1}{2}a_{-n,...,-2,0}(a_{-n,...,-2,-1})^{-\frac{1}{2}}$.
 And for $p_{-1}=2$ we obtain

\begin{align*}
&-\frac{1}{2}a_{-n,...,-2,1}(a_{-n,...,-2,-1})^{-\frac{1}{2}}-\frac{1}{4}(a_{-n,...,-2,0})^2(a_{-n,...,-2,-1})^{-\frac{3}{2}}=\\
&=-\frac{1}{2}(a_{-n,...,-2,-1})^{-\frac{3}{2}}(a_{-n,...,-2,1}a_{-n,...,-2,-1}+\frac{1}{2}(a_{-n,...,-2,0})^2)=0,
\end{align*}

where we used relation form Lemma \ref{soot2} below.

Now consider the case of an arbitrary highest vector.  It can be represented as follows

$$
v_0=v'_0(a_{-n,...,-2,-1})^{\frac{1}{2}},
$$

where  $v'_0$ is a polynomial in determinants.  An arbitrary vector  $f$
 is a linear combination of  $
F_{0,-1}^{p_{-1}}F_{-1,-2}^{p_{-2}}...F_{-n+1,-n}^{p_{-n}}v_0$. This vector is of type
\eqref{f12}.

Now we must prove that every vector of type \eqref{f12},
which satisfies the condition of Lemma  \ref{lfmf} is a solution of the indicator system.

Operators $L_{-i,-i+1}$, $i=n,...,2$ act onto determinants of orders $n-i+1$ that is onto determinants of orders
$1,...,n-1$. Such determinants occur in $f_1$, $f_2$, in particular
they occur in non-negative integer powers, the sum of  powers of
determinants of order $i$ equals to $r_{-n+i-1}$ since Lemma
\ref{lfmf} holds. Thus conditions $L_{-i,-i+1}^{r_{-i}+1}f=0 $ for
$i=n,...,2$ hold.

Now consider the equation
$L_{-1,0}^{2m_{-1}+1}f=L_{-1,0}^{2[m_{-1}]+1+1}f=0$,  where
$[m_{-1}]$ is an integer part. The operator  $L_{-1,0}^{2[m_{-1}]+1+1}$ acts according to the Leibnitz ruler onto each summand in
\eqref{f12} as follows.

 Either $L_{-1,0}^{2[m_{-1}]+1+1}$ acts onto the second factor  $f_1$ or $f_2$. Then we obtain  $0$, since the sum of powers of determinants of order $n$ in  $f_1$ and   $f_2$  equals  $[m_{-1}]$, such functions are annihilated by $L_{-1,0}^{2[m_{-1}]+1}$.

 Either  $L_{-1,0}^{2[m_{-1}]+1}$  acts onto the second factor  $f_1$ or $f_2$,  and
   $L_{-1,0}$  acts onto the first factor. We obtain  $0$  by the same reason.

 Either $L_{-1,0}^{2[m_{-1}]+2-k}$   acts onto the second factor, and $L_{-1,0}^{k}$ acts onto the first factor, where $k\geq 2$.But the first factor is a vector of a representation with the highest weight $[\frac{1}{2},\frac{1}{2}]$,  thus it is annihilated by $L_{-1,0}^{2}$.

Thus  \eqref{f12} vanishes under the action of
$L_{-1,0}^{2m_{-1}+1}$. In the case of the series   $B$ the Lemma is proved.

Now consider the case of the series  $D$. Let  $m_{-1}\geq 0$.

Suppose that the highest vector is
$(a_{-n,...,-2,-1})^{\frac{1}{2}}$. The an arbitrary vector of the
representation is a linear combination of vectors of type
$F_{1,-2}^{p_{-1}}F_{-1,-2}^{p_{-2}}...F_{-n+1,-n}^{p_{-n}}(a_{-n,...,-2,-1})^{\frac{1}{2}}$.
This vector is non-zero only if  $p_{-2}=...=p_{-n}=0$ and
$p_{-1}=0$ or  $1$. Indeed  when $p_{-1}=1$ we obtain the vector
$a_{-n,...,-1,1}(a_{-n,...,-2,-1})^{-\frac{1}{2}}$. And when
$p_{-1}=2$ we obtain

\begin{align*}
&a_{-n,...,1,2}(a_{-n,...,-2,-1})^{-\frac{1}{2}}-(a_{-n,...,-1,1})^2(a_{-n,...,-2,-1})^{-\frac{3}{2}}=\\
&(a_{-n,...,-2,-1})^{-\frac{3}{2}}(a_{-n,...,1,2}a_{-n,...,-2,-1}-(a_{-n,...,-1,1})^2)=0.
\end{align*}

In the derivation of this formulas relations from Lemma  \ref{soot2} were used.  The further considerations in the case $D$ and $m_{-1}\geq 0$ are analogous to considerations in the case   $B$. The case $D$ with $m_{-1}<0$ is considered analogously.

\end{proof}

\begin{lem}
In the case of the series  $D$ for  $m_{-1}<0$ the analogues of Lemmas  \ref{l1}, \ref{l2} take place.
 But we must change the determinants of order  $n$ by the ruler
  $a_{i_1,...,i_n}\mapsto
\bar{a}_{\{i_1,...,i_n\}\mid_{-1\leftrightarrow 1}}$, where
$.\mid_{-1\leftrightarrow 1}$ is an interchange of indices  $-1$ and
$1$.
\end{lem}

\begin{defn}
Functions that satisfy conditions of Lemmas \ref{l1}, \ref{l2} we call {\it the admissible functions of the determinants.}
\end{defn}

Since determinant and functions of them satisfy automatically the condition 1 of Theorem \ref{tzh}, we obtain the following corollary

\begin{lem}
The  admissible functions of the determinants are exactly the fucntions that form an irreducible representation with the highest vector \eqref{stv}.
\end{lem}

 \subsection{ $g_{n-1}$-highest vectors}

 %Действиетльно, если это не так, то оператор из $g_{n-1}$, соответствующие положительным корням, получаем многочлен, не равный нулю. Но старший относительно  $g_{n-1}$ вектор должен обрать

 \begin{lem}
\label{lemma1}
A vector that is highest with respect to  $g_{n-1}$
can be represented as an admissible function of determinants such
that
 %\begin{enumerate}
 %\item
 it depends on determinants  $a_{i_1,...,i_k}$ that  vanish under the action of  elements corresponding to positive roots of  $g_{n-1}$.
 %\item One has   $L_{-i,-i}f=m_{-i}f$, $i=n,...,1$.
 %\item The function satisfies the indicator system.
 %\end{enumerate}

\end{lem}

\proof
%Since the highest vector is a function of determinants then an arbitrary element of an irreducible representation is also  a function of determinants.

Let us show that one can represent a vector as  a function of determinants that are highest with resect to   $g_{n-1}$.
Let us use a realization in function on  $Z$  and the Proposition \ref{claizh}.
  %, удовлетворяющих индикаторной системе (см. \cite{zh2}).
  %Функции на всей группе $G_n$, удовлетворяющие условиям из раздела \ref{fugu},
  %естественно ограничиваются на  $Z$  и это ограничение есть изоморфизм между двумя реализациями.

In \cite{zh2} it is shown that in this realization the vectors that are highest
 with respect to $g_{n-1}$ are polynomials in matrix elements  $z_{-k,-1}$, $z_{-k,1}$\footnote{In \cite{zh2} the cases $GL_{n+1}$, $Sp_{2n}$ are considered, the cases  $O_{2n}$, $O_{2n+1}$ are considered analogously }.  Admissible function of determinants that are highest with respect to   $g_{n-1}$ (see Section  \ref{oprst}, where such determinants are listed),  under restriction to  $Z$  give all possible such functions. By Proposition \ref{claizh}  the restriction to $Z$ is a bijection. Hence in the realization in the functions on  $G_n$  every highest with respect to   $g_{n-1}$ vector is  an admissible function of determinants that are highest with respect to  $g_{n-1}$.

%The fact that the conditions $2$, $3$ select exactly the space of  all $g_{n-1}$-highest vectors follows from Theorem \ref{tzh}.

\endproof

Let us give an explicit description of the functions selected by Lemma \ref{lemma1}.

\subsubsection{Determinants that are highest with respect to  $g_{n-1}$}
\label{oprst} Using formulas \eqref{edet1}, \eqref{fbd},
\eqref{fd} we obtain that determinants that are highest with respect to $g_{n-1}$  are
\begin{align*}
a_{-n}, a_{\pm 1},\,\,\,\,\, a_{-n,-n+1},a_{-n,\pm
1},a_{-1,1},\,\,...\,\,,a_{-n,...-3,-2},a_{-n,...-3,\pm
1},a_{-n,...-4,-1,1}\,\,\,\,\,a_{-n,...-3,-1,1},
\end{align*}
and

\begin{enumerate}
\item  $a_{-n,...-2,-1}$, $a_{-n,...-2,1}$, $a_{-n,...-3,-1,1}$ in the case of the series $A$,
\item $a_{-n,...-2,-1}$, $a_{-n,...-2,1}$, $a_{-n,...-3,-1,1}$, $a_{-n,...-3,-2,2}$ in the case of the series $C$,
\item $a_{-n,...-2,-1}$, $a_{-n,...-2,1}$, $a_{-n,...-3,-1,1}$, $a_{-n,...-2,0}$ in the case of the series  $B$.
\item  $a_{-n,...-3,2}$, $\bar{a}_{-n,...-2,-1}$, $\bar{a}_{-n,...-2,1}$, $\bar{a}_{-n,...-3,-1,1}$, $\bar{a}_{-n,...-3,-1,2}$,
 $\bar{a}_{-n,...-3,-2,2}$, $\bar{a}_{-n,...-3,1,2}$  в
in the case of the series  $D$ and  $m_{-1}<0$.
\item  $a_{-n,...-3,2}$, $a_{-n,...-2,-1}$, $a_{-n,...-2,1}$, $a_{-n,...-3,-1,1}$, $a_{-n,...-3,-1,2}$, $a_{-n,...-3,-2,2}$,
 $a_{-n,...-3,1,2}$  in the case of the series  $D$ and  $m_{-1}\geq 0$.
\end{enumerate}

But no all of these  functions are independent.

\begin{lem}
\label{ssp4}
 For functions on $Sp_{2n}$ one has a relation $a_{-n,....,-3,-2,2}=-a_{-n,...,-3,-1,1}$
\end{lem}
\proof

 For a typical matrix  $X\in Sp_{4}$ one has a Gauss decomposition \eqref{gaus}.
The matrices  $\zeta$, $\delta$ and  $z$ can be represented as exponents of Lie algebra elements $\zeta=e^{A}$, $\delta=e^B$, $z=e^C$. The matrix  $A$ is a linear combination of  $F_{i,j}$, $i>j$,  $B$ is a linear combination of   $F_{i,i}$, $C$  is a linear combination of    $F_{i,j}$, $i<j$. Using this fact we obtain  a parametrization of matrices $A$, $B$, $C$, then after taking an exponent a parametrization of matrices $\zeta$, $\delta$ и  $z$, and finally a presentation of an arbitrary matrix  $X\in Sp_{4}^0$. Then we can check the equality $a_{-n,....,-3,-2,2}=-a_{-n,...,-3,-1,1}$ by direct computations \footnote{Of course it is better to do it using a computer}. Since   $Sp_{4}^0$ is dense in    $Sp_{4}$ the equality holds everywhere on  $Sp_{4}$.

Consider the case of an arbitrary group  $Sp_{2n}$.   Let  $\mathcal{F}_{i,j}(\alpha)$  be a matrix with units on the diagonal, with   $\alpha$ on the place  $(i,j)$  and with  $-\alpha$ on the place $(-j,-i)$.  The matrices  $\mathcal{F}_{i,j}(\alpha)$ belong to $Sp_{2n}$. A multiplication by this matrix on the right is equivalent to  doing an elementary transformation of the matrix $X$: we add to the  $j$-th row the $i$-th row with a coefficient  $\alpha$, and we add simultaneously to the $-i$-th row the $-j$-the row with a coefficient $-\alpha$.  A multiplication by $\mathcal{F}_{i,j}(\alpha)$  on the left is equivalent to an analogous transformation of rows.
% Отметим, что выполнение такого преобразования со строками  $-т,-4,...,-n$  не меняет обоих определителей  $a_{-n,....,-3,-2,2}$ и $a_{-n,...,-3,-1,1}$.

Multiplying by  $\mathcal{F}_{i,j}(\alpha)$ on the left and on the right we can transform   $X$ preserving  $a_{-n,....,-3,-2,2}$ and $a_{-n,...,-3,-1,1}$ to the following form

\begin{equation}
\begin{pmatrix}
x_{-n,-n}&...&0& 0 & 0& 0 &0& 0 &...& x_{-n,n}\\
...\\
0&...&x_{-3,-3}& 0 & 0 &0 &0& 0& ...&0\\
0&...&0& x_{-2,-2} & x_{-2,-1} & x_{-2,1} & x_{-2,2}&0&...&0\\
0&...&0& x_{-1,-2} & x_{-1,-1} & x_{-1,1} & x_{-1,2}&0&...&0\\
0&...&0& x_{1,-2} & x_{1,-1} & x_{1,1} & x_{1,2}&0&...&0\\
0&...&0& x_{2,-2} & x_{2,-1} & x_{2,1} & x_{2,2}&0&...&0\\
0&...&x_{3,-3}& 0 & 0 & 0 &0&x_{3,3}&...&0\\
...\\
x_{n,-n}&...&0& 0 & 0 & 0 &0&0&...&x_{n,n}\\
\end{pmatrix}
\end{equation}

Since matrices  $\mathcal{F}_{i,j}(\alpha)$  belong to $Sp_{2n}$, this matrix belongs to  $Sp_{2n}$.  We have

\begin{align*}
&a_{-n,....,-3,-2,2}=x_{-n,-n}...x_{-3,-3}\cdot det\begin{pmatrix}  x_{-2,-2} & x_{-2,2}\\ x_{-1,-2} &x_{-1,2}\end{pmatrix},\\
&a_{-n,....,-3,-1,1}=x_{-n,-n}...x_{-3,-3}\cdot det\begin{pmatrix} x_{-2,-1} & x_{-2,1}\\ x_{-1,-1} &x_{-1,1}\end{pmatrix}.
\end{align*}

The submatrix

\begin{equation}
\begin{pmatrix}

x_{-2,-2} & x_{-2,-1} & x_{-2,1} & x_{-2,2}\\
 x_{-1,-2} & x_{-1,-1} & x_{-1,1} & x_{-1,2}\\
 x_{1,-2} & x_{1,-1} & x_{1,1} & x_{1,2}\\
 x_{2,-2} & x_{2,-1} & x_{2,1} & x_{2,2}\\
\end{pmatrix}
\end{equation}

 belongs to $Sp_{4}$. Using the equality  $a_{-2,2}=-a_{-1,1}$ for  $Sp_{4}$ we obtain that   $a_{-n,....,-3,-2,2}$ и $a_{-n,...,-3,-1,1}$ в $Sp_{2n}$.

\endproof
Analogously one prove the following statement.
\begin{lem}
\label{soot2} In the case of the series   $B$  one has an equality
$$a_{-n,...-2,1}a_{-n,...,-2,-1}=-\frac{1}{2}a_{-n,...-2,0}^2,$$

in the case of the series  $D$ and $m_{-1}<0$ one has  equalities the
\begin{align*}
&a_{-n,...-3,2}\cdot a_{-n,...-3,-2}=-a_{-n,...-3,1}\cdot a_{-n,...-3,-1},\,\,\,\, \bar{a}_{-n,...-2,-1}=0,\,\,\,\, \\
&\bar{a}_{-n,...-3,-1,1}\cdot a_{-n,...-3,-2}=-a_{-n,...-3,1}\cdot \bar{a}_{-n,...-3,-2,1},\\
& \bar{a}_{-n,...-3,-1,2}\cdot
a_{-n,...-3,-2}^2=-a_{-n,...-3,1}^2\cdot
\bar{a}_{-n,...-3,-2,1},\,\,\,\,\\&
\bar{a}_{-n,...-3,-2,2}=-a_{-n,...-3,-1,1},\,\,\,\,
\bar{a}_{-n,...-3,1,2}=0,
\end{align*}

in the case of the series  $D$ and $m_{-1}\geq 0$
 one has the
equalities
\begin{align*}
&a_{-n,...-3,2}\cdot a_{-n,...-3,-2}=-a_{-n,...-3,1}\cdot a_{-n,...-3,-1},\,\,\,\, a_{-n,...-2,1}=0,\,\,\,\, \\
&a_{-n,...-3,-1,1}\cdot a_{-n,...-3,-2}=-a_{-n,...-3,1}\cdot a_{-n,...-3,-2,-1},\\
& a_{-n,...-3,-1,2}=0,\,\,\,\,
a_{-n,...-3,-2,2}=a_{-n,...-3,-1,1},\,\,\,\,\\
&a_{-n,...-3,1,2}\cdot a_{-n,...-3,-2}^2=-a_{-n,...-3,1}^2\cdot a_{-n,...-3,-2,-1}
\end{align*}
\end{lem}

So we have proved the following.

\begin{thm}
\label{le5} The vectors that are highest with respect to $g_{n-1}$
are admissible functions of variables
\begin{align*}
a_{-n}, a_{\pm 1},\,\,\,\,\, a_{-n,-n+1},a_{-n,\pm
1},a_{-1,1},\,\,...\,\,,a_{-n,...-3,-2},a_{-n,...-3,\pm
1},a_{-n,...-4,-1,1},
\end{align*}
and also

\begin{enumerate}
\item $a_{-n,...-2,-1}$, $a_{-n,...-2,1}$, $a_{-n,...-3,-1,1}$ in the cases  $A$ and $C$,
\item  $a_{-n,...-2,-1}$, $a_{-n,...-2,0}$  in the case $B$.

\item   $a_{-n,...-3,2}$, $\bar{a}_{-n,...-2,1}$, $\bar{a}_{-n,...-3,-1,1}$, $\bar{a}_{-n,...-3,-1,2}$  в
in the case  $D$  and negative $m_{-1}<0$.

\item $a_{-n,...-3,2}$, $a_{-n,...-2,-1}$,  $a_{-n,...-3,-1,1}$, $a_{-n,...-3,1,2}$  в
in the case  $D$  and  positive $m_{-1}\geq 0$.
\end{enumerate}

\end{thm}

In the case  $D$  we use determinants that are dependent, since if
we express the dependent determinants through others we obtain
non-admissible  functions.

\subsubsection{Relations between determinants.}

Since we consider determinants of submatrices of a big matrix, we have the Plucker relations

\begin{lem}
\label{sootpl}
$\sum_{\sigma} a_{\sigma(i_1,...,i_k}a_{j_1,...,j_l)}=0,$ we take a summation over all  permutations of the set of indices $i_1,...,i_k,j_1,...,j_l$.
\end{lem}

If one passes to the  realization in the space of the functions on the group $Z$,  then one can express all determinants through the independent matrix elements $z_{-k,-1},z_{-k,1}$ $k=-n,...,-2$,
  and  $z_{-1,1}$ in the cases  $A,C$; $z_{-1,0}$ in the case  $B$. One has
  \begin{align*}
  &a_{-n,...,-k}=1,,\,\, a_{-n,...,-k-1,1}=z_{-k,1},,\,\, a_{-n,...,-k-1,-1}=z_{-k,-1},\,\,\\
  & a_{-n,...,-k-2,-1,1}=z_{-k,1}z_{-k-1,-1}-z_{-k,1}z_{-k-1,-1},\\
  &a_{-n,...,-2,0}=z_{-1,0} \text{ for the series  $B$},\\
&a_{-n,...,-3,2}=-z_{-2,1}z_{-2,-1} \text{ for the series $D$},\\
  &\bar{a}_{-n,...,-2,1}=1,,\,\, \bar{a}_{-n,...-3,-1,1}=z_{-2,-1},\,\,\bar{a}_{-n,...-3,-1,2}=-z_{-2,-1}^2 \text{ for the series  $D$ and $m_{-1}<0$},\\
   &a_{-n,...,-2,-1}=1,,\,\,a_{-n,...-3,-1,1}=-z_{-2,1},\,\,a_{-n,...-3,-1,2}=-z_{-2,1}^2 \text{ for the series $D$ and $m_{-1}\geq 0$
}.\\
  \end{align*}

  Analyzing these expressions we come to the conclusion

\begin{lem}
\label{sootful} There are no other relations other then those, presented in Lemmas
\ref{soot2}, \ref{sootpl}.
\end{lem}

 Using results of the Section \ref{oprst} and  Theorem  \ref{le5}
we come to the conclusion.

\begin{thm}
\label{lempoc}  A function represents a $g_{n-1}$-highest vector in a representation with the highest vector \eqref{stv}  if and only it it is a polynomial in determinants listed in Theorem  \ref{le5} or a function of these determinants of type  \eqref{f12}  such that the following property holds.  If one expands this functionas into a sum of products of determiants  the sum of powers of determinants
$a_{-n,...,-k-1,-k}$, $a_{-n,...,-k-1,-1}$,
  $a_{-n,...,-k-1,1}$, $a_{-n,...,-k-2,-1,1}$, and $a_{-n,...,-3,2}$ in the case $k=2$ and the series is $D$, equals  $r_{-k}$ for  $k=n,...,2$
  and also

\begin{enumerate}
\item  For the series  $A$, $C$  the sum of powers  $a_{-n,...,-2,-1}$, $a_{-n,...,-2,1}$,  $a_{-n,...,-3,-1,1}$ equals  $m_{-1}$.
\item  For the series  $B$  the sum of powers    $a_{-n,...,-2,-1}$, $a_{-n,...,-2,0}$,  $a_{-n,...,-3,-1,1}$  equals  $m_{-1}$.
\item  For the series  $D$ and  $m_{-1}<0$  the sum of powers
 $\bar{a}_{-n,...,-2,1}$,   $\bar{a}_{-n,...,-3,-1,1}$,  $\bar{a}_{-n,...,-3,-1,2}$ equals  $-m_{-1}$.
\item  For the series  $D$% and $m_{-1}\geq 0$
    the sum of powers
 $a_{-n,...,-2,-1}$,   $a_{-n,...,-3,-1,1}$,  $a_{-n,...,-3,1,2}$ equals  $m_{-1}$.
\end{enumerate}

\end{thm}

%\begin{proof}
%Admissible functions described above satisfy the equations $L_{-i,-i}f=m_{-i}f$. Using formulas for the action of   $L_{-i,-j}$ we obtain that the indicator system is also satisfied. We need only to consider the case  $B$ and the equation   $L_{-1,0}^{2m_{-1}+1}f=0$ from the indicator system.

%The operator  $L_{-1,0}$ acts only on the set of upper indices of type  $\{-n,...,-2,-1\}$,  the action is given by the formula  $.\mid_{-1\mapsto 0}-.\mid_{0\mapsto 1}$. Thus,  $L_{-1,0}^{2m_{-1}+1}f=0$ means that the sum of powers of determinants of order  $n$ is less or equal than $m_{-1}$. Thus the admissible functions for the series  $B$ described above do satisfy the indicator system.
%\end{proof}

\section{The problem of restriction $g_2\downarrow g_1$}
\label{abcd2}

\subsection{The case  $\mathfrak{sp}_4$}
\label{ac2}
Let us be given a representation of  $\mathfrak{sp}_4$ with the highest weight
$[m_{-2},m_{-1}]$.Consider the problem of restriction
$\mathfrak{sp}_4\downarrow\mathfrak{sp}_2$. In  \cite{zh2} it is shown that the problems
$\mathfrak{gl}_3\downarrow\mathfrak{gl}_1$ and
$\mathfrak{sp}_{4}\downarrow\mathfrak{sp}_{2}$ are equivalent. Thus
$\mathfrak{sp}_2$-highest vectors are encoded by integer  tableaux  that satisfy the betweeness conditions

\begin{align}
\begin{split}
\label{ggc3}
&m_{-2}\,\,\,\,\,\,\, m_{-1} \,\,\,\,\,\,\,0\\
&\,\,\,\,\,\,\,\,\,\,\,\,k_{-2}\,\,\,\,\,\,\,\,\,   k_{-1}   \,\,\,\,\,\,\,\,\,\,\,\,\,\,\,\\
&\,\,\,\,\,\,\,\,\,\,\,\,\,\,\,\,\,\,\,\,\,\,\,\,s_{-2}
\end{split}
\end{align}

Let us give a formula for a function corresponding to a tableau. In
  \cite{bb}  it is shown that to a $\mathfrak{gl}_{2}$-highest vector  encoded by a tableau \eqref{ggc3} with $k_{-2}=s_{-2}$ there corresponds a polynomial

\begin{equation}
a_{1}^{m_{-2}-k_{-2}}a_{-2}^{k_{-2}-m_{-1}}a_{-2,1}^{m_{-1}-k_{-1}}a_{-2,-1}^{k_{-2}}.
\end{equation}

To obtain a vector corresponding to the tableau \eqref{ggc3} one must apply to this monomial  $\frac{(s_{-2}-k_{-1})!}{(k_{-2}-k_{-1})!}E_{-1,-2}^{k_{-2}-s_{-2}}$, one obtains modulo multiplication by a constant the polynomial

\begin{equation}
\label{ac} a_{1}^{m_{-2}-k_{-2}}a_{-2,-1}^{k_{-2}}\sum
\frac{1}{p_{-1}!p_{-1,1}!p_{1}!p_{-2,1}!}
a_{-1}^{p_{-1
}}a_{-1,1}^{p_{-1,1}}a_{-2}^{p_{1}}a_{-2,1}^{p_{-2,1}},
\end{equation}
 where a summation is taken over all non-negative $p_{-1},p_{-1,1},p_{1},p_{-2,1}$, such that

 \begin{equation}
 p_{-1}+p_{-2}=k_{-2}-m_{-1},\,\,\,p_{-1,1}+p_{-2,1}=m_{-1}-k_{-1},\,\,\, p_{-1}+p_{-1,1}=k_{-2}-s_{-2}.
 \end{equation}

\subsection{The case $\mathfrak{o}_5$}
Let us be given  a representation of  $\mathfrak{o}_5$ with the highest weight $[m_{-2},m_{-1}]$.
Consider a problem  restriction
$\mathfrak{o}_5\downarrow\mathfrak{o}_3$. In the Zhelobenko's realization
the problem $\mathfrak{o}_{2n+1}\downarrow \mathfrak{o}_{2n-1}$  is considered in \cite{a2},
 where a relation with the restriction problem
$\mathfrak{gl}_{n+1}\downarrow\mathfrak{gl}_{n-1}$ is established.
The $\mathfrak{o}_3$-highest vectors in a
$\mathfrak{o}_5$-representation are encoded by a number $\sigma$ and
a  tableaux that satisfy the betweeness conditions whose elements
are simultaneously integers of half-integers

\begin{align}
\begin{split}
\label{dio5}
&\,\,\,\,\,\,m_{-2}\,\,\,\,\,\,\, m_{-1} \,\,\,\,\,\,\,0\\
&\sigma,\,\,\,\,\,\, \,\,\,\,k_{-2}\,\,\,\,\,\,\,\,\,   k_{-1}   \,\,\,\,\,\,\,\,\,\,\,\,\,\,\,\\
&\,\,\,\,\,\,\,\,\,\,\,\,\,\,\,\,\,\,\,\,\,\,\,\,\,\,\,\,\,\,s_{-2}
\end{split}
\end{align}
where  $\sigma=0,1$. If $k_{-1}=0$ then  $\sigma=0$.

In  \cite{a2} the realization in functions on the subgroup  $Z$ is used. If one passes to the realization  in functions on the whole group one obtains

\begin{equation}
\label{bbo}
a_{-2,0}^{\sigma}a_{1}^{m_{-2}-k_{-2}}a_{-2,-1}^{k_{-2}}\sum
\frac{1}{p_{-1}!p_{-1,1}!p_{1}!p_{-2,1}!} a_{-1}^{p_{-1
}}a_{-1,1}^{p_{-1,1}}a_{-2}^{p_{-2}}(a_{-2,1})^{p_{-2,1}},
\end{equation}
where a summation is taken over all  non-negative integers $p_{-1},p_{-1,1},p_{1},p_{-2,1}$, such that

 \begin{equation}
 p_{-1}+p_{-2}=k_{-2}-m_{-1},\,\,\,p_{-1,1}+p_{-2,1}=m_{-1}-k_{-1},\,\,\, p_{-1}+p_{-1,1}=k_{-2}-s_{-2}.
 \end{equation}

To obtain a function described in Theorem \ref{lempoc} one must make a change $a_{-2,1}\mapsto -\frac{a_{-2,0}^2}{2a_{-2,-1}}$.

Let us prove this fact without a reference to  \cite{a2}.

Consider the case of an integer highest weight. Then admissible functions are just polynomials in determinants.
%Consider polynomials in  determinants that satisfy   \ref{tzh}, \ref{lempoc}.
Consider first the case of an integer highest weight.
   Since $f$ is an element of an irreducible representation with the highest vector  $v_0$, it can be represented as a linear combination of vectors of type $F_{-1,-2}^{p}F_{0,-1}^{q}v_0$, where $v_0=a_{-2}^{m_{-2}-m_{-1}}a_{-2,-1}^{m_{-1}}$.
One has \begin{equation}\label{f01}F_{0,-1}^2a_{-2,-1}=2F_{0,-1}a_{-2,-1}a_{-2,0}=2a_{-2,0}^2-2a_{-2,-1}a_{-2,1}=-4a_{-2,-1}a_{-2,1},\end{equation} thus in the case  $q=2q'$ one has

$$F_{-1,-2}^{p}F_{0,-1}^{2q'}v_0=const E_{-1,-2}^pE_{1,-1}^{q'}v_0.$$

The polynomial on the right satisfies the conditions of Theorem
\ref{lempoc} for the algebra $\mathfrak{gl}_3$ and the highest
weight $[m_{-2},m_{-1},0]$.  Thus there exists an isomorphism
between the span of $F_{-1,-2}^{p}F_{0,-1}^{2q'}v_0$ and the space
of representation of $\mathfrak{gl}_3$ with the highest weight
$[m_{-2},m_{-1},0]$. Using the previous Section one concludes that
in the span of $F_{-1,-2}^{p}F_{0,-1}^{2q'}v_0$ there exists a base
\eqref{bbo}, given by tableaux \eqref{dio5} where  $\sigma=0$.

In the considered case the eigenvalues of  $F_{-2,-2}$ and   $F_{-1,-1}$
correspond to eigenvalues of   $E_{-2,-2}$ and  $E_{-1,-1}-E_{1,1}$. The
later are equal to   $s_{-2}$ and  $-2(k_{-2}+k_{-1})+(m_{-2}+m_{-1})+s_{-2}$.

 In the case  $q=2q'+1$ one has

 $$F_{-1,-2}^{p}F_{0,-1}^{2q'+1}(a_{-2}^{m_{-2}-m_{-1}}a_{-2,-1}^{m_{-1}})=const (E_{-1,-2}^pE_{1,-1}^{q'}(a_{-2}^{m_{-2}-m_{-1}}a_{-2,-1}^{m_{-1}-1}))a_{-2,0}.$$

If one removes  $a_{-2,0}$ then one obtains on the right a
polynomial which satisfies the conditions of Theorem  \ref{lempoc}
for the algebra $\mathfrak{gl}_3$ and the highest weight
$[m_{-2}-1,m_{-1}-1,0]$.
 In the space of such polynomials there exists a base  \eqref{ac} encoded by an integer tableaux of type

 \begin{align}
 \begin{split}
\label{ggoc3}
&m_{-2}-1\,\,\,\,\,\,\, m_{-1}-1 \,\,\,\,\,\,\,0\\
&\,\,\,\,\,\,\,\,\,\,\,\,k_{-2}-1\,\,\,\,\,\,\,\,\,   k_{-1}-1   \,\,\,\,\,\,\,\,\,\,\,\,\,\,\,\\
&\,\,\,\,\,\,\,\,\,\,\,\,\,\,\,\,\,\,\,\,\,\,\,\,s_{-2}-1
\end{split}
\end{align}
Note that here  $k_{-1}-1\geq 0$

 Thus in the span of  $F_{-1,-2}^{p}F_{0,-1}^{2q'+1}v_0$ there exists a base \eqref{bbo} encoded by  \eqref{dio5} where  $\sigma=1$ and $k_{-1}\geq 1$.

 The eigenvalues of $F_{-2,-2}$ and   $F_{-1,-1}$ correspond to eigenvalues of  $E_{-2,-2}-1$ and  $E_{-1,-1}-E_{1,1}$. The later are equal to  $s_{-2}-1$ and $-2(k_{-2}-1+k_{-1}-1)+(m_{-2}-1+m_{-1}-1)+s_{-2}-1=-2(k_{-2}+k_{-1})+(m_{-2}+m_{-1})+s_{-2}+1$.

% Итак, для вектора \eqref{bbo} получаем, что  $s_{-2}$ есть  $(-2)$-компонента веса соответствующего вектора, а  $(-1)$-компонента вычисляется по формуле  $-2(k_{-2}+k_{-1})+(m_{-2}+m_{-1})+s_{-2}+\sigma$.

Thus in the case of integer highest weight there exists a base \eqref{bbo} encoded by  \eqref{dio5} where $\sigma=0,1$.

Now consider the case of half-integer highest weight. One has
 \begin{equation}\label{e38}F_{0,-1}a_{-2,-1}^{1/2}=\frac{1}{2}a_{-2,0}a_{-2,-1}^{-1/2},\,\,\,\, F_{0,-1}^{2}a_{-2,-1}^{1/2}=0.\end{equation}
The highest vector can be written as $v_0=a_{-2}^{m_{-2}-m_{-1}}a_{-2,-1}^{[m_{-1}]+1/2}$, where  $[m_{-1}]$ is the integer part.   A vector of an irreducible representation is a linear combination of vectors

$$F_{-1,-2}^{p}F_{0,-1}^{q}(a_{-2}^{m_{-2}-m_{-1}}a_{-2,-1}^{[m_{-1}]+1/2}).$$

If  $q=2q'$ then using  \eqref{f01}, \eqref{e38} we obtain that that this vector equals to
 \begin{align*}
&F_{-1,-2}^{p}F_{0,-1}^{2q'}(a_{-2}^{m_{-2}-m_{-1}}a_{-2,-1}^{[m_{-1}]+1/2})=const E_{-1,-2}^{p}E_{1,-1}^{q'}(a_{-2}^{m_{-2}-m_{-1}}a_{-2,-1}^{[m_{-1}]})a_{-2,-1}^{1/2}
,\end{align*}

thus we have a natural isomorphism between the span of  $F_{-1,-2}^{p}F_{0,-1}^{2q'}v_0$ and the space of an irreducible representation of  $\mathfrak{gl}_3$ with the highest weight
$[m_{-2}-\frac{1}{2},m_{-1}-\frac{1}{2},0]$. In this space there exists a base encoded by an integer tableau of type

\begin{align}
 \begin{split}
\label{gggoc3}
&m_{-2}-\frac{1}{2}\,\,\,\,\,\,\, m_{-1}-\frac{1}{2} \,\,\,\,\,\,\,0\\
&\,\,\,\,\,\,\,\,\,\,\,\,k_{-2}-\frac{1}{2}\,\,\,\,\,\,\,\,\,   k_{-1}-\frac{1}{2}   \,\,\,\,\,\,\,\,\,\,\,\,\,\,\,\\
&\,\,\,\,\,\,\,\,\,\,\,\,\,\,\,\,\,\,\,\,\,\,\,\,s_{-2}-\frac{1}{2}
\end{split}
\end{align}

Using the formula for the vector corresponding to a tableau in the case $\mathfrak{gl}_3$ we obtain that in the span of
$F_{-1,-2}^{p}F_{0,-1}^{2q'}v_0$  the exist a base  \eqref{ac} encoded by half-integer tableau  \eqref{dio5} with  $\sigma=0$.

The eigenvalues of  $F_{-2,-2}$ and $F_{-1,-1}$
correspond to eigenvalues of  $E_{-2,-2}-\frac{1}{2}$ and
$E_{-1,-1}-E_{1,1}-\frac{1}{2}$. One can easily check that $s_{-2}$
is the  $(-2)$-th component of the weight and
$(-1)$-th component of the weight is calculated as
$-2(k_{-2}+k_{-1})+(m_{-2}+m_{-1})+s_{-2}$.

Now let  $q=2q'+1$, one has

\begin{align*}
&F_{-1,-2}^{p}F_{0,-1}^{2q'+1}
(a_{-2}^{m_{-2}-m_{-1}}a_{-2,-1}^{[m_{-1}]+1/2})=const (
E_{-1,-2}^{p}E_{1,-1}^{q'}
a_{-2}^{m_{-2}-m_{-1}}a_{-2,-1}^{[m_{-1}]})a_{-2,0}a_{-2,-1}^{-1/2}
,\end{align*}

thus there exists an isomorphism between the span of
$F_{-1,-2}^{p}F_{0,-1}^{2q'+1}v_0$ and the space of an irreducible representation of
$\mathfrak{gl}_3$ with the highest weight
$[m_{-2}-\frac{1}{2},m_{-1}-\frac{1}{2},0]$. Thus in the span the exists a base of type \eqref{ac} encoded by  half-integer tableau \eqref{dio5} with $\sigma=1$. The eigenvalues of  $F_{-2,-2}$ and  $F_{-1,-1}$ correspond to eigenvalues of $E_{-2,-2}-\frac{1}{2}$ and
$E_{-1,-1}-E_{1,1}+\frac{1}{2}$. One can easily check that $s_{-2}$
is the  $(-2)$-th component of the weight and
$(-1)$-th component of the weight is calculated as
$-2(k_{-2}+k_{-1})+(m_{-2}+m_{-1})+s_{-2}+1$.

Thus in the case of half-nteger highest weight there exists a base \eqref{bbo} encoded by  \eqref{dio5} where $\sigma=0,1$.

The weight is calculated by the following ruler: $s_{-2}$ is a  $(-2)$-component of the weight of the  \eqref{bbo} and it's  $(-1)$-th component is equals to $-2(k_{-2}+k_{-1})+(m_{-2}+m_{-1})+s_{-2}+\sigma$.

\subsection{The case $\mathfrak{o}_4$}

Let us be given an irreducible representation of $\mathfrak{o}_4$ with the highest weight $[m_{-2},m_{-1}]$, where $m_{-2}\geq -|m_{-1}|$.
%  Так как алгебра $\mathfrak{o}_2$
%одномерна, то задача ограничения $\mathfrak{o}_4\downarrow
%\mathfrak{o}_2$ эквивалента задаче построения базиса в
%$\mathfrak{o}_4$-представлении.

%\subsubsection{Задача ограничения  $\mathfrak{o}_4 \downarrow \mathfrak{o}_2$}\label{indd}

Let us construct a base using a restriction of algebras
$\mathfrak{o}_4 \downarrow \mathfrak{o}_2$. One has
$\mathfrak{o}_4=\mathfrak{o}_2\oplus\mathfrak{o}_2$, where these
copies of  $\mathfrak{o}_2$ are

 \begin{align*}
 &span<F_{-1,-2},F_{-2,-2}-F_{-1,-1},F_{-2,-1}>,\\
 &span<F_{1,-2},F_{-2,-2}+F_{-1,-1},F_{-2,1}>.
 \end{align*}

%Можно предполагать, что  $m_{-1}\geq 0$. Действительно, если  $m_{-1}<0$, то данное представление можно получить из представления  $[m_{-2},-m_{-1}]$ подкруткой действия
%$\mathfrak{o}_4$ на  автоморфизм $\mathfrak{o}_4$, переставляющий   два экземпляра  $\mathfrak{o}_2$. Таким образом, при решении задачи описания базиса можно предполагать, что $m_{-1}\geq 0$.

There exists the following base in the representation

\begin{align}
\begin{split}
\label{kl}
%&\frac{(m_{-2}-m_{-1}-k)!}{(m_{-2}-m_{-1}-k)!}
&\frac{(m_{-2}-m_{-1}-k)!}{(m_{-2}-m_{-1})!}\frac{(m_{-2}+m_{-1}-l)!}{(m_{-2}+m_{-1})!}F_{-1,-2}^kF_{1,-2}^lv_0,\\
&0\leq k\leq m_{-2}-m_{-1},\,\,\,\,\,0\leq l\leq  m_{-2}+m_{-1},\,\,\, k,l\in\mathbb{Z},
\end{split}
\end{align}

and  $v_0$ is a highest vector.  The weight of this vector equals to
\begin{equation}
\label{wesvectora}
(m_{-2}-k-l,m_{-1}+k-l)
\end{equation}

Let us give another indexation of vectors  \eqref{kl}.

\subsubsection{The case $m_{-1}\geq 0$}
%Let $m_{-1}\geq 0$.
 The highest vector is written as follows
\begin{equation}
\label{stmpl2} v_0=a_{-2}^{m_{-2}-m_{-1}}{a}_{-2,-1}^{m_{-1}},
\end{equation}

The operator $F_{-1,-2}$ acts onto determinants that are highest with respect to   $\mathfrak{o}_2$ as follows

\begin{align*}
&a_{-2}\mapsto a_{-1},\,\,\, a_{1}\mapsto -a_{2},
&\text{  other determinant }\mapsto 0.
\end{align*}

The operator  $F_{1,-2}$ acts onto determinants that are highest with respect to   $\mathfrak{o}_2$ as follows

\begin{align*}
&a_{-2}\mapsto a_{1},\,\,\, a_{-1}\mapsto -a_{2},\,\,\, a_{-2,-1}\mapsto -2a_{-1,1},\,\,\,a_{-1,1}\mapsto a_{1,2},
&\text{ other determinant }\mapsto 0.
\end{align*}

Thus
 \eqref{kl} modulo multiplication on a constant
 equals to
 % is a linear combination of monomials of type

% $$
%a_{-2}^{p_{-2}}a_{-1}^{p_{-1}}a_{1}^{p_{1}}a_{2}^{p_{2}}a_{-2,-1}^{p_{-2,-1}}{a}_{-1,1}^{p_{-1,1}}{a}_{1,2}^{p_{1,2}}a_{-1,2}^{p_{-1,2}}
% $$

%then explicitely  \eqref{kl} equals to

% \begin{align*}
%&\sum
% %\frac{1}{p_{-2}!p_{-1}!p_{1}!p_{-2,-1}!p_{-1,1}!p_{1,2}!}
%\frac{1}{p_{1}!p_{2}!(p_{-1,1}+p_{1,2})!p_{1,2}!}
% a_{-2}^{p_{-2}}a_{-1}^{p_{-1}}a_{1}^{p_{1}}a_{2}^{p_{2}}a_{-2,-1}^{p_{-2,-1}}{a}_{-1,1}^{p_{-1,1}}{a}_{1,2}^{p_{1,2}},\\
%& p_{-2}+p_{-1}+p_{1}=m_{-2}-m_{-1},\,\,\,\, p_{-2,-1}+p_{-1,1}+2p_{1,2}=m_{-1}.\\
%&p_{-1}+p_{2}=l,\,\,\,\, p_{2}+p_{-1}+p_{-1,1}+p_{1,2}=k
%\end{align*}

 \begin{align*}
&\sum
 %\frac{1}{p_{-2}!p_{-1}!p_{1}!p_{-2,-1}!p_{-1,1}!p_{1,2}!}
%\frac{1}{p_{1}!p_{2}!(p_{-1,1}+p_{1,2})!p_{1,2}!}
\frac{(-1)^{p'_{2}+p''_{2}+p_{-1,1}+p_{1,2}}}{p_{-1}!p_{1}!p'_{2}!p''_{2}!(p_{-1,1}+p_{1,2})!p_{1,2}!}
 a_{-2}^{p_{-2}}a_{-1}^{p_{-1}}a_{1}^{p_{1}}a_{2}^{p'_{2}+p''_{2}}a_{-2,-1}^{p_{-2,-1}}{a}_{-1,1}^{p_{-1,1}}{a}_{1,2}^{p_{1,2}},\\
&\text{ where the powers satisfy the equalities } \\
& p_{-2}+p_{-1}+p_{1}+p'_{2}+p''_{2}=m_{-2}-m_{-1},\,\,\,\, p_{-2,-1}+p_{-1,1}+2p_{1,2}=m_{-1}.\\
&p_{-1}+p'_{2}=k,\,\,\,\, p_{1}+p''_{2}+p_{-1}+p_{-1,1}+p_{1,2}=l
\end{align*}

The powers $p_{1}, p_{-1},p'_{2},p''_{2},p_{-2}, p_{-1,1}, p_{1,2}$ are integer and non-negative

By Lemma \ref{soot2} one can express all determinants through  $a_{-2}$, $a_{1}$, $a_{-1}$, $a_{-2,-1}$, one obtains that \eqref{kl} modulo multiplication by a constant can be rewritten as follows

\begin{equation}
a_{-2}^{m_{-2}-m_{-1}-k-l}a_{1}^la_{-1}^ka_{-2,-1}^{m_{-1}},
\end{equation}

but here
% это выражение  не является полиномом, так как
the power of   $a_{-2}$ can be negative.

Define numbers   $k_{-2},s_{-2}$ by formulas

\begin{equation}
\label{klopr1} k_{-2}=m_{-2}-k,\,\,\, s_{-2}=m_{-2}-k-l.
\end{equation}

%\begin{prop}\label{pq}
%Числа $p$ и $q$ восстанавливаются по следующему правилу.

%\begin{enumerate}
%\item Если $m'_{-2,1}-m_{-1,2}>0$, то $p=m_{-2,2}-m'_{-2,2}$, $q=m_{-2,2}-m'_{-2,2}+m'_{-2,1}-m_{-1,2}$.
%\item Если $m'_{-2,1}-m_{-1,2}\leq 0$, то $q=m_{-2,2}-m'_{-2,2}$, $p=m_{-2,2}-m'_{-2,2}-m'_{-2,1}+m_{-1,2}$.
%\end{enumerate}
%\end{prop}

Let us compose an integer or half-integer  tableau % \eqref{gco4}.

\begin{align}\begin{split}\label{gco4}
& \,\,\,m_{-2} \,\,\,\, m_{-1}   \,\,\,\, \\
&  \,\,\, \,\,\, \,\,\,\,\,k_{-2}\\
&\,\,\,\,\,\,\,\,\,\,\,\,\,s_{-2}\end{split}
\end{align}

For it's elements the following restrictions hold\footnote{Thus for $s_2$ in  \eqref{gco4}  the betweness condition doe not hold}

\begin{align}
\begin{split}
\label{nero42}
& m_{-2}\geq k_{-2}\geq m_{-1}\\
& m_{-2}\geq |s_{-2}|
\end{split}
\end{align}

Using \eqref{wesvectora}, we obtain the following statement
%Имеет место Предложение %, аналогичное \ref{stroka041}:

\begin{prop}\label{stroka041} $(-2)$-component of the weight of the vector encoded by \eqref{gco4}, equals
$s_{-2}$,  $(-1)$-component of the weight  equals
$-2k_{-2}+(m_{-2}+m_{-1})+s_{-2}$
\end{prop}

%\subsubsection{Резюме}

%\begin{lem}
%Во всех рассмотренных случаях были построены диаграммы, обладающие
%следующим свойством.  В случаях серий  $B$, $C$ и $D$ при условии  $m_{-1}<0$ если взять многочлен, соответствующий
%диаграмме, то имеет вид  $a_{1}^{m_{-2}-k_{-2}}f$ и  $f$ не делится на $a_{1}$. В случае же   $D$  при условии  $m_{-1}<0$ то же верно, если взять степень  $a_{-1}$ вместо  $a_{1}$.
%\end{lem}

\subsubsection{The case $m_{-1}<0$}
%Suppose that $m_{-1}<0$.

The highest vector is written as

\begin{equation}
\label{stmpl11} v_0=a_{-2}^{m_{-2}-m_{-1}}\bar{a}_{-2,1}^{-m_{-1}},
\end{equation}

The  vector \eqref{kl}
%%в случае
%%положительного  $m_{-1}$
modulo multiplication by a constant equals

\begin{align*}
&\sum
%\frac{1}{p_{-2}!p_{-1}!p_{1}!p_{-2,-1}!p_{-1,1}!p_{1,2}!}
\frac{(-1)^{p'_{2}+p''_{2}+p_{-1,1}+p_{-1,2}}}{p'_{2}!p''_{2}!p_{-1}!(p_{-1,1}+p_{-1,2})!p_{-1,2}!}
a_{-2}^{p_{-2}}a_{-1}^{p_{-1}}a_{1}^{p_{1}}a_{2}^{p'_{2}+p''_{2}}\bar{a}_{-2,1}^{p_{-2,1}}\bar{a}_{-1,1}^{p_{-1,1}}\bar{a}_{-1,2}^{p_{-1,2}},
\end{align*}
where summation is taken over non-negative integers such that
 \begin{align*}
& p_{-2}+p_{-1}+p_{1}+p'_{2}+p''_{2}=m_{-2}-m_{-1},\,\,\,\, p_{-2,1}+p_{-1,1}+2p_{-1,2}=-m_{-1}.\\
&p_{1}+p'_{2}=k,\,\,\,\, p''_{2}+p_{-1}+p_{-1,1}+p_{-1,2}=l
\end{align*}

Due to relations from Lemma  \ref{soot2} all determinants can be expressed though   $a_{-2}$, $a_{1}$, $a_{-1}$,
$\bar{a}_{-2,1}$, as a result we obtain an expression

\begin{equation}
a_{-2}^{m_{-2}-m_{-1}-k-l}a_{1}^la_{-1}^k\bar{a}_{-2,1}^{-m_{-1}},
\end{equation}

%но это выражение  не является полиномом, так как степень  $a_{-2}$ может быть %отрицательной.

Put

\begin{equation}
\label{klopr11} k_{-2}=m_{-2}-l,\,\,\, s_{-2}=m_{-2}-k-l.
\end{equation}

 Compose an integer or half-integer tableau

\begin{align}\begin{split}\label{gco41}
& \,\,\,m_{-2} \,\,\,\, -m_{-1}   \,\,\,\, \\
&  \,\,\, \,\,\, \,\,\,\,\,k_{-2}\\
&\,\,\,\,\,\,\,\,\,\,\,\,\,s_{-2}\end{split}
\end{align}

The following inequalities hold

\begin{align}
\begin{split}
\label{nero411}
& m_{-2}\geq k_{-2}\geq -m_{-1}\\
& m_{-2}\geq |s_{-2}|
\end{split}
\end{align}

This is the needed indexation.

Using  \eqref{wesvectora} we obtain the following statement.  Note that we give a formula not  $(-1)$,  but for the $(+1)$-th component of the weight.

\begin{prop}\label{stroka042} $(-2)$-th component of a weight of a vector encoded by \eqref{gco41} equals to
$s_{-2}$,  $(+1)$-th component of the weight equals to
$2k_{-2}-(m_{-2}-m_{-1})-s_{-2}$
\end{prop}

One can say that in the considered case $1$ and $-1$ change their
roles.

\section{The extension of restriction problems $\mathfrak{gl}_{3}\downarrow\mathfrak{gl}_{1}\supset\mathfrak{gl}_{2}\downarrow\mathfrak{gl}_{0}$}

In this Section we investigate a relation of restriction problems
$\mathfrak{gl}_{2}\downarrow\mathfrak{gl}_{0}$ and
$\mathfrak{gl}_{3}\downarrow\mathfrak{gl}_{1}$.  This relation describes  the fact that the Gelfand-Tsetlin tableaux  that appears in the problem $\mathfrak{gl}_{3}\downarrow\mathfrak{gl}_{1}$
are obtained from tableaux for the problem
$\mathfrak{gl}_{2}\downarrow\mathfrak{gl}_{0}$ by extension to the left.

Thus consider the problem of restriction
%$\mathfrak{gl}_{2}\downarrow\mathfrak{gl}_{0}$,  that is the problem of construction of the Gelfand-Tsetlin  base  for a representation
%$\mathfrak{gl}_{2}$ with the highest weight $[m_{-1},0]$.
%
%
%
%The base solutions are
%\begin{align*}
%a_{1}^{p_{-1}}a_{-1}^{m_{-1}-p_{-1}},\,\,\,\,\, 0\leq p_{-1}\leq m_{-1}.\end{align*}  They are encoded by
%
%
%
%\begin{align}\begin{split}
%\label{diagl2}
%&m_{-1}\,\,\,\,\,\,\,0\\
%&\,\,\,\,\,\,k_{-1}\,\,\,\,\,\,\,\,\,\,\,\,\,\,\,,\text{  where } k_{-1}=m_{-1}-p_{-1}.
%\end{split}\end{align}
%
%Now consider the problem
$\mathfrak{gl}_{3}\downarrow\mathfrak{gl}_{1}$  and consider a tableau \begin{align}\begin{split}\label{gc3}
&m_{-2}\,\,\,\,\,\,\,m_{-1}\,\,\,\,\,\,\,0\\
&\,\,\,\,\,\,\,\,\,\,\,\,k_{-2}\,\,\,\,\,\,k_{-1}\,\,\,\,\,\,\,\,\,\,\,\,\,\,\,\\
&\,\,\,\,\,\,\,\,\,\,\,\,\,\,\,\,\,\,\,\,\,\,\,\,s_{-2}
\end{split}
\end{align}

According to the Section \ref{ac2}, to this tableau there corresponds a polynomial

\begin{equation*}
 a_{1}^{m_{-2}-k_{-2}}a_{-2,-1}^{k_{-1}}\sum
\frac{1}{p_{-1}!p_{-1,1}!p_{1}!p_{-2,1}!}
a_{-1}^{p_{-1
}}a_{-1,1}^{p_{-1,1}}a_{-2}^{p_{1}}a_{-2,1}^{p_{-2,1}},
\end{equation*}
where a summation is taken over all non-negative integers $p_{-1},p_{-1,1},p_{1},p_{-2,1}$, such that

 \begin{equation}
 p_{-1}+p_{-2}=k_{-2}-m_{-1},\,\,\,p_{-1,1}+p_{-2,1}=m_{-1}-k_{-1},\,\,\, p_{-1}+p_{-1,1}=k_{-2}-s_{-2}.
 \end{equation}

\begin{lem}
Take a span of products of type

\begin{align*}
&a_{-2}^{p_{-2}}a_{-1}^{p_{-1}}a_{1}^{p_{1}}a_{-2,-1}^{p_{-2,-1}}a_{-2,1}^{p_{-2,1}}a_{-1,1}^{p_{-1,1}},\\
&p_{-2}+p_{-1}+p_{1}=m_{-2}-m_{-1},\,\,\,\,\,p_{-2,-1}+p_{-2,1}+p_{-1,1}=m_{-1},
\end{align*}
with a fixed sum $p_{-2,1}+p_{-1,1}$, introduce a number $k_{-1}$ by a formula $p_{-2,1}+p_{-1,1}=m_{-1}-k_{-1}$. Then in this span there exists a base, indexed by tableaux

\begin{align*}
&m_{-2}\,\,\,\,\,\,\,\,\,m_{-1}\\
&\,\,\,\,\,\,\,\,\,\,\,\,k_{-2}\,\,\,\,\,\,   k_{-1}   \,\,\,\,\,\,\,\,\,\,\,\,\,\,\,\\
&\,\,\,\,\,\,\,\,\,\,\,\,\,\,\,\,\,\,\,\,\,\,\,\,s_{-2}
\end{align*}
%where  $\mathcal{D}$  is a tableau of type  \eqref{diagl2}with a  highest weight $[m_{-1},0]$ and a fixed number   $k_1$.

\end{lem}

The space described in Lemma  is actually a span of $\mathfrak{gl}_3$-tableaux with a
fixed  $\mathfrak{gl}_{2}$-tableau in the right upper corner. Thus
we call the elements of this space solutions of {\it the problem of
extension }
$\mathfrak{gl}_{3}\downarrow\mathfrak{gl}_{1}\supset\mathfrak{gl}_{2}\downarrow\mathfrak{gl}_{0}$
with given $m_{-2},m_{-1},k_{-1}$.

 Since for the considered monomials the power  $p_{-2,-1}=k_{-1}$  is fixed we can divide all the polynomials by   $a_{-2,-1}^{p_{-2,-1}}$. We obtain the following statement

\begin{lem}

\label{lemoch}
Take a span of products of type

\begin{align*}
&a_{-2}^{p_{-2}}a_{-1}^{p_{-1}}a_{1}^{p_{1}}a_{-2,1}^{p_{-2,1}}a_{-1,1}^{p_{-1,1}},\\
&p_{-2}+p_{-1}+p_{1}=m_{-2}-m_{-1},\,\,\,\,\,p_{-2,1}+p_{-1,1}=m_{-1}-k_{-1},
\end{align*}
 Then in this span there exists a base, indexed by tableaux

\begin{align*}
&m_{-2}\,\,\,\,\,\,\,\,\,m_{-1}\\
&\,\,\,\,\,\,\,\,\,\,\,\,k_{-2}\,\,\,\,\,\,   k_{-1}   \,\,\,\,\,\,\,\,\,\,\,\,\,\,\,\\
&\,\,\,\,\,\,\,\,\,\,\,\,\,\,\,\,\,\,\,\,\,\,\,\,s_{-2}
\end{align*}
%where  $\mathcal{D}$  is a tableau of type  \eqref{diagl2}with a  highest weight $[m_{-1},0]$ and a fixed number   $k_1$.

\end{lem}

 The space in which a base in Lemma \ref{lemoch} is constructed is defined  not by numbers  $m_{-2}$, $m_{-1}$, $k_{-1}$, but by differences  $m_{-2}-m_{-1}$, $m_{-1}-k_{-1}$, which must be integer and non-negative. That is why we allow below the numbers  $m_{-2}$, $m_{-1}$, $k_{-1}$ to be simultaneously half-integer.

\section{The extension problems $g_{n}\downarrow g_{n-1}\supset g_{2}\downarrow g_{1}$}

\label{abcdn}
For simplicity put  $n=3$. The case of an arbitrary  $n$  is discussed at the end of the Section.
 The space of $g_2$-highest vectors in the problem  $g_3\downarrow g_2$ is spanned by products of type

\begin{align}
\label{rein}
&f=a_{-3}^{p_{-3}}a_{-1}^{p_{-1}}a_{1}^{p_1}a_{-3,-2}^{p_{-3,-2}}a_{-3,1}^{p_{-3,1}}a_{-1,1}^{p_{-1,1}}\cdot f_2,
\end{align}

where

\begin{align*}
&f_2=a_{-3,-2,-1}^{p_{-3,-2,-1}}a_{-3,-1,1}^{p_{-3,-1,1}}a_{-3,-2,1}^{p_{-3,-2,1}}\text{  для серий $A$,$C$ },\\
&f_2=a_{-3,-2,-1}^{p_{-3,-2,-1}}a_{-3,-1,1}^{p_{-3,-1,1}}a_{-3,-2,0}^{p_{-3,-2,0}}\text{  для серии $B$ },\\
%&f_2=a_{-3,2}^{p'_{-3,2}+p''_{-3,2}}\bar{a}_{-3,-2,-1}^{p_{-3,-2,-1}}\bar{a}_{-3,-1,1}^{p_{-3,-1,1}}
%\bar{a}_{-3,-1,2}^{p_{-3,-1,2}}\text{  для серии $D$ и }m_{-1}<0\\
%&f_2=a_{-3,2}^{p'_{-3,2}+p''_{-3,2}}a_{-3,-2,-1}^{p_{-3,-2,-1}}a_{-3,-1,1}^{p_{-3,-1,1}}
%a_{-3,1,2}^{p_{-3,1,2}}\text{  для серии $D$ и }m_{-1}\geq 0
&f_2=a_{-3,2}^{p_{-3,2}}\bar{a}_{-3,-2,-1}^{p_{-3,-2,-1}}\bar{a}_{-3,-1,1}^{p_{-3,-1,1}}
\bar{a}_{-3,-1,2}^{p_{-3,-1,2}}\text{  для серии $D$ и }m_{-1}<0\\
&f_2=a_{-3,2}^{p_{-3,2}}a_{-3,-2,-1}^{p_{-3,-2,-1}}a_{-3,-1,1}^{p_{-3,-1,1}}
a_{-3,1,2}^{p_{-3,1,2}}\text{  для серии $D$ и }m_{-1}\geq 0
\end{align*}

The conditions for exponents are written in the Theorem
\ref{lempoc}. Below we write them explicitly.

\subsection{The case of series  $A$, $C$, $B$}
\label{acde}
In the case of series  $A$, $C$ the powers are non-negative and  satisfy

\begin{align*}
&p_{-3}+p_{-1}+p_{1}=m_{-3}-m_{-2},\,\,\,\, p_{-3,-2}+p_{-3,-1}+p_{-3,1}+p_{-1,1}=m_{-2}-m_{-1},\\
& p_{-3,-2,-1}+p_{-3,-1,1}+p_{-3,-2,1}=m_{-1},
\end{align*}

and in the case   $B$ the powers are non-negative and  satisfy

\begin{align*}
&p_{-3}+p_{-1}+p_{1}=m_{-3}-m_{-2},\,\,\,\, p_{-3,-2}+p_{-3,-1}+p_{-3,1}+p_{-1,1}=m_{-2}-m_{-1},\\
& p_{-3,-2,-1}+p_{-3,-1,1}+p_{-3,-2,0}=m_{-1},
\end{align*}

Consider solutions such that

\begin{align}
\begin{split}
\label{usl}
&p_{-3,1}+p_{-1,1}=m_{-2}-k_{-2},\\
&\text{$p_{-3,-1},p_{-3,-2},p_{-3,-2,-1},p_{-3,-1,1},p_{-3,-2,1}$ where  $p_{-3,-2,0}$  is fixed }
\end{split}
\end{align}

For the series   $A$, $ B$, $C$ the mapping of   \eqref{rein}  to a polynomial that is a solution of the extension
$\mathfrak{gl}_{3}\downarrow\mathfrak{gl}_{1}\supset\mathfrak{gl}_{2}\downarrow\mathfrak{gl}_{0}$ with numbers $m_{-3}, m_{-2},k_{-2}$ by the ruler

%\begin{align}
%\begin{split}
%\label{mainsoot}
%&a_{1}^{p_{1}}\mapsto a_1^{p_{1}},\,\,\,\,\, a_{-1}^{p_{-2}}\mapsto a_{-1}^{p_{-2}},\,\,\,\,\,a_{-3}^{p_{-3}}\mapsto a_{-2}^{p_{-3}},\\
%& a_{-1,1}^{p_{-1,1}}\mapsto a_{-1,1}^{p_{-1,1}}, \,\,\,\,\, a_{-3,1}^{p_{-3,1}}\mapsto a_{-2,1}^{p_{-3,1}}
%\end{split}
%\end{align}

\begin{align}
\begin{split}
\label{mainsoot} &f\mapsto
a_{-2}^{p_{-3}}a_1^{p_{1}}a_{-1}^{p_{-2}}a_{-2,1}^{p_{-3,1}}a_{-1,1}^{p_{-1,1}}
\end{split}
\end{align}
%\begin{align}
%\begin{split}
%\label{mainsoot}
%&a_{1}^{p_{1}}\mapsto a_1^{p_{1}},\,\,\,\,\, a_{-1}^{p_{-2}}\mapsto a_{-1}^{p_{-2}},\,\,\,\,\,a_{-3}^{p_{-3}}\mapsto a_{-2}^{p_{-3}},\\
%& a_{-1,1}^{p_{-1,1}}\mapsto a_{-1,1}^{p_{-1,1}}, \,\,\,\,\, a_{-3,1}^{p_{-3,1}}\mapsto a_{-2,1}^{p_{-3,1}}
%\end{split}
%\end{align}

gives a well-defined mapping from the space of polynomials \eqref{rein},
that satisfy \eqref{usl}, into the space of solutions of the problem
$\mathfrak{gl}_{3}\downarrow\mathfrak{gl}_{1}\supset\mathfrak{gl}_{2}\downarrow\mathfrak{gl}_{0}$.
To prove this we need to check that  \eqref{mainsoot} respects relations between determinants. The problem is that in $f$  on the left side of the \eqref{mainsoot} determinant for the problem
 $g_3\downarrow g_{2}$ occur and on the right side of  \eqref{mainsoot}  determinants for the problem $\mathfrak{gl}_{3}\downarrow\mathfrak{gl}_{1}$ occur.
 The relations between determinants are written in Lemma  \ref{sootful}.  It says that the the determinants that are not mapped into unit under \eqref{mainsoot}
 satisfy only Plucker relations, which take place for determinant on both sides.

One sees that \eqref{mainsoot}    maps isomorphically into the space of admissible linear combinations of \eqref{rein},
 that satisfy \eqref{usl}, to the space of solutions of the extension
 $\mathfrak{gl}_{3}\downarrow\mathfrak{gl}_{1}\supset\mathfrak{gl}_{2}\downarrow\mathfrak{gl}_{0}$ with $m_{-3},m_{-2},k_{-2}$. Thus we obtain.

 \begin{lem}
 \label{lemabc}
 In the space of solutions \eqref{rein}, that satisfy  \eqref{usl} there exists a base indexed by an integer or half-integer tableau

\begin{align}
\begin{split}
\label{aaa}
&m_{-3}\,\,\,\,\,\,\,\,\,m_{-2}\\
&\,\,\,\,\,\,\,\,\,\,\,\,k_{-3}\,\,\,\,\,\,   k_{-2}   \,\,\,\,\,\,\,\,\,\,\,\,\,\,\,\\
&\,\,\,\,\,\,\,\,\,\,\,\,\,\,\,\,\,\,\,\,\,\,\,\,s_{-3}
\end{split}
\end{align}
 %где  $\mathcal{D}$  - диаграмма для  $g_2$ со старшим весом  $[m_{-2},m_{-1},0]$.
  %Диаграмма определяется  показателями $q_{-1}=p_{-3,-1},q_{-2}=p_{-3,-2},q=_{-2,-1}=p_{-3,-2,-1},q_{-1,1}=p_{-3,-1,1},q_{-2,1}=p_{-3,-2,1}$ или  $q_{-2,0}=p_{-3,-2,0}$  и заданным выше числом  $k_1$.
 \end{lem}

If one does not put conditions \eqref{usl} but allows these indices
to take all possible values one obtains

 \begin{cor}
 \label{cabc}
 In the space of solutions of  \eqref{rein},
% удовлетворяющих условию \eqref{usl}
 there exists a base indexed by an integer or half-integer tableau of type

\begin{align}
\begin{split}
\label{bbd}
&m_{-3}\,\,\,\,\,\,\,\\
&\,\,\,\,\,\,\,\,\,\,\,\,k_{-3}\,\,\,\,\,\,   \mathcal{D}   \,\,\,\,\,\,\,\,\,\,\,\,\,\,\,\\
&\,\,\,\,\,\,\,\,\,\,\,\,\,\,\,\,\,\,\,\,\,\,\,\,s_{-3}
\end{split}
\end{align}
 where  $\mathcal{D}$ is a   $g_2$-tableau with the highest weight $[m_{-2},m_{-1},0]$ in the case of the series  $C$, $B$ and    $[m_{-2},m_{-1},0]$ in the case of the series $A$.
 \end{cor}

  This tableau is a span of  \eqref{aaa} with fixed $m_{-3},k_{-3},s_{-3}$, but different indices $p_{-3,-1},p_{-3,-2}...$ which were fixed before.
   More precise put   $q_{1}=p_{-3,1}+p_{-1,1}$, $q_{-1}=p_{-3,-1}$, $q_{-2}=p_{-3,-2}$, $q_{-2,-1}=p_{-3,-2,-1}$, $q_{-1,1}=p_{-3,-1,1}$, $q_{-2,1}=p_{-3,-2,1}$ in the case of the series  $A$, $C$ or   $q_{-2,0}=p_{-3,-2,0}$ in the case of the series  $B$. To the vector  \eqref{aaa} associate a  monomial
\begin{align}
  \begin{split}
  \label{bbb}
  &b_{1}^{q_{1}}b_{-1}^{q_{-1}}b_{-2}^{q_{-2}}b_{-2,-1}^{q_{-2,-1}}b_{-1,1}^{q_{-1,1}}b_{-2,1}^{q_{-2,1}}\text{ in the cases  $A$, $C$},\\
  &b_{1}^{q_{1}}b_{-1}^{q_{-1}}b_{-2}^{q_{-2}}b_{-2,-1}^{q_{-2,-1}}b_{-1,1}^{q_{-1,1}}b_{-2,0}^{q_{-2,0}}\text{  in the case  $B$},
  \end{split}
  \end{align}

In the span of such monomial there exist a base given by $g_{2}$-tableaux   $\mathcal{D}$.  Then the vector \eqref{bbd} is a linear combination of \eqref{aaa} with the same coefficients as we take for
 monomials \eqref{bbb} to obtain a polynomial corresponding to a tableau.

\subsection{The case  $D$.}

\subsubsection{The case $m_{-1}\geq 0$}

%Let $m_{-1}\geq 0$.
In the considered case the following inequalities hold

\begin{align}
\begin{split}
\label{reshd}
&p_{-3}+p_{-1}+p_{1}=m_{-3}-m_{-2},\,\,\,\, p_{-3,-2}+p_{-3,-1}+p_{-3,1}+p_{-1,1}+p_{-3,2}=m_{-2}-m_{-1},\\
& p_{-3,-2,-1}+p_{-3,-1,1}+2p_{-3,1,2}=m_{-1}.
\end{split}
\end{align}

All power  are non-negative.

There exist relations between determinants. % To present a well-defined description of solutions of the extension of restriction problems
Let us remove dependent determinants.
 By Lemma \ref{soot2}  we can remove $a_{-3,2}, a_{-3,-1,1}, a_{-3,1,2}$,  then the powers of
$a_{-3,-2}$, $a_{-3,-1}$, $a_{-3,1}$, $a_{-3,-2,-1}$ are changed by the following rulers

\begin{enumerate}
\item The power of  $a_{-3,-2}$ is changed to  $q_{-3,-2}=p_{-3,-2}-p_{-3,2}-p_{-3,-1,1}-2p_{-3,-1,2}$.
\item  The power of $a_{-3,-1}$  is changed to  $q_{-3,-1}=p_{-3,-1}+p_{-3,2}$.
\item The power of   $a_{-3,1}$  is changed to  $q_{-3,1}=p_{-3,1}+p_{-3,2}+p_{-3,-1,1}+2p_{-3,-1,2}$.
\item The power of   $a_{-3,-2,-1}$  is changed to  $q_{-3,-2,-1}=m_{-1}$.
\end{enumerate}

Note that

$$
q_{-3,-2}+q_{-3,-1}+q_{-3,1}=m_{-2}-m_{-1},
$$

 the power $q_{-3,-2}$ can become negative, but $q_{-3,-1}$, $q_{-3,1}$ are positive.

Consider solutions such that

\begin{align}
\begin{split}
\label{usld}
&q_{-3,-1}+p_{-1,1}=m_{-2}-k_{-2},\\
&q_{-3,1},q_{-3,-2},q_{-3,-2,-1},\text{  are fixed}
\end{split}
\end{align}

From one hand one has
$$
m_{-2}-k_{-2}=q_{-3,-1}+p_{-1,1}\geq 0,\,\,\,\,\Rightarrow m_{-2}\leq k_{-2},
$$

from the other hand one has

$$
k_{-2}=m_{-2}-p_{-3,-1}-p_{-3,2}-p_{-1,1}=m_{-1}+p_{-3,-2}+p_{-3,1}\geq
m_{-1}.
$$
Thus $k_{-2}$ satisfies the inequalities

$$m_{-2}\geq k_{-2}\geq m_{-1}.$$

 The same arguments as in the previous Section show that
the mapping

%\begin{align}
%\begin{split}
%\label{mainsootd1}
%&a_{1}^{p_{1}}\mapsto a_1^{p_{1}},\,\,\,\,\, a_{-1}^{p_{-1}}\mapsto a_{-1}^{p_{-1}},\,\,\,\,\,a_{-3}^{p_{-3}}\mapsto a_{-2}^{p_{-3}},\\
%& {a}_{-1,1}^{q_{-1,1}}\mapsto a_{-1,1}^{q_{-1,1}}, \,\,\,\,\,
%a_{-3,1}^{q_{-3,1}}\mapsto a_{-2,1}^{q_{-3,1}}
%\end{split}
%\end{align}

\begin{align}
\begin{split}
\label{mainsootd1}
& f\mapsto a_1^{p_{1}}a_{-1}^{p_{-1}}a_{-2}^{p_{-3}}a_{-1,1}^{p_{-1,1}}a_{-2,1}^{q_{-3,1}}
\end{split}
\end{align}

sends isomorphically the space of functions \eqref{reshd}  that satisfy \eqref{usld}, into the space
 of solutions of the extension $\mathfrak{gl}_{3}\downarrow\mathfrak{gl}_{1}\supset\mathfrak{gl}_{2}\downarrow\mathfrak{gl}_{0}$
  with $m_{-3},m_{-2},k_{-2}$. Thus one obtains the following statement.

\begin{lem}
 In the case $m_{-1}\geq 0$  in the space of functions \eqref{rein}, that satisfy  \eqref{usld} there exists a base indexed by integer or half-integer tableaux

\begin{align}
\begin{split}
\label{ddd}
&m_{-3}\,\,\,\,\,\,\,\,\,m_{-2}\\
&\,\,\,\,\,\,\,\,\,\,\,\,k_{-3}\,\,\,\,\,\,   k_{-2}   \,\,\,\,\,\,\,\,\,\,\,\,\,\,\,\\
&\,\,\,\,\,\,\,\,\,\,\,\,\,\,\,\,\,\,\,\,\,\,\,\,s_{-3}
\end{split}
\end{align}
 \end{lem}

Analogously to  Corollary \ref{cabc} for series  $A$, $B$, $C$,  for the series $D$ and $m_{-1}\geq 0$ we obtain

 \begin{lem}
 In the case $m_{-1}\geq 0$  in the space of functions \eqref{reshd} that satisfy \eqref{usld} there exists a base given by integer or half-integer tableaux of type

\begin{align*}
&m_{-3}\,\,\,\,\,\,\,\\
&\,\,\,\,\,\,\,\,\,\,\,\,k_{-3}\,\,\,\,\,\,   \mathcal{D}   \,\,\,\,\,\,\,\,\,\,\,\,\,\,\,\\
&\,\,\,\,\,\,\,\,\,\,\,\,\,\,\,\,\,\,\,\,\,\,\,\,s_{-3}
\end{align*}
 where  $\mathcal{D}$ is a tableau for  $\mathfrak{o}_4$ with the highest weight
 $[m_{-2},m_{-1}]$.
 % Диаграмма определяется  показателями
 %$q_{-1}=p_{-3,-1},q_{-2}=p_{-3,-2},q=_{-2,-1}=p_{-3,-2,-1},q_{-1,1}=p_{-3,-1,1},q_{-2,1}=p_{-3,-2,1}$ или
 %  $q_{-2,0}=p_{-3,-2,0}$  и заданным выше числом  $k_1$.
 \end{lem}

\subsubsection{The case $m_{-1}< 0$}

In the case $m_{-1}<0$ we operate as follows.

In the case under consideration the following inequalities do hold

\begin{align}
\begin{split}
\label{reshd2}
&p_{-3}+p_{-1}+p_{1}=m_{-3}-m_{-2},\,\,\,\,\\& p_{-3,-2}+p_{-3,-1}+p_{-3,1}+p_{-1,1}+p'_{-3,2}+p''_{-3,2}=m_{-2}-m_{-1},\\
& p_{-3,-2,-1}+p_{-3,-1,1}+2p_{-3,-1,2}=-m_{-1}.
\end{split}
\end{align}

All powers are non-negative.

Using relations from Lemma \ref{soot2}  we can express determinants $a_{-3,2}, a_{-3,-1,1}, a_{-3,-1,2}$ through other determinants, the powers of determinants
$a_{-3,-2}$, $a_{-3,-1}$, $a_{-3,1}$, $a_{-3,-2,1}$ are changed by the following ruler.

\begin{enumerate}
\item The power of  $a_{-3,-2}$ is changed to  $q_{-3,-2}=p_{-3,-2}-p_{-3,2}-p_{-3,-1,1}-2p_{-3,-1,2}$.
\item  The power of   $a_{-3,1}$ is changed to  $q_{-3,1}=p_{-3,1}+p_{-3,2}$.
\item The power of  $a_{-3,-1}$ is changed to $q_{-3,-1}=p_{-3,-1}+p_{-3,2}+p_{-3,-1,1}+2p_{-3,-1,2}$.
\item The power of   $a_{-3,-2,1}$ is changed to  $q_{-3,-2,1}=-m_{-1}$.
\end{enumerate}

Note that
$$
q_{-3,-2}+q_{-3,-1}+q_{-3,1}=m_{-2}-m_{-1}.
$$

The power  $q_{-3,-2}$ can be negative but  $q_{-3,-1}$, $q_{-3,1}$ is non-negative.

Consider functions such that

\begin{align}
\begin{split}
\label{usld2}
&q_{-3,1}+p_{-1,1}=m_{-2}-k_{-2},\\
&q_{-3,-1},q_{-3,-2},q_{-3,-2,1},\text{   are fixed }
\end{split}
\end{align}
Then

$$m_{-2}\geq k_{-2}\geq -m_{-1}.$$

The same arguments as in the  previous Section show that  the mapping

%\begin{align}
%\begin{split}
%\label{mainsootd1}
%& a_{1}^{p_{1}}\mapsto a_1^{p_{1}},\,\,\,\,\, a_{-1}^{p_{-1}}\mapsto a_{-1}^{p_{-1}},\,\,\,\,\,a_{-3}^{p_{-3}}\mapsto a_{-2}^{p_{-3}},\\
%& {a}_{-1,1}^{q_{-1,1}}\mapsto a_{-1,1}^{q_{-1,1}}, \,\,\,\,\,
%a_{-3,1}^{q_{-3,1}}\mapsto a_{-2,1}^{q_{-3,1}}
%\end{split}
%\end{align}

\begin{align}
\begin{split}
\label{mainsootd2}
& f\mapsto a_1^{p_{1}}a_{-1}^{p_{-1}}a_{-2}^{p_{-3}}a_{-1,1}^{p_{-1,1}}a_{-2,1}^{q_{-3,-1}}
\end{split}
\end{align}

 is an isomorphism between the space of polynomials  \eqref{reshd} that satisfy \eqref{usld} into
  the space of solution of the extension problem $\mathfrak{gl}_{3}\downarrow\mathfrak{gl}_{1}\supset\mathfrak{gl}_{2}\downarrow\mathfrak{gl}_{0}$ with $m_{-3},m_{-2},k_{-2}$.
We obtain the following Statements.

\begin{lem}
 When $m_{-1}<0$ in the space of functions \eqref{rein}, that satisfy \eqref{usld} there exists a base indexed by integer or half-integer  tableaux

\begin{align}
\begin{split}
\label{ddd2}
&m_{-3}\,\,\,\,\,\,\,\,\,m_{-2}\\
&\,\,\,\,\,\,\,\,\,\,\,\,k_{-3}\,\,\,\,\,\,   k_{-2}   \,\,\,\,\,\,\,\,\,\,\,\,\,\,\,\\
&\,\,\,\,\,\,\,\,\,\,\,\,\,\,\,\,\,\,\,\,\,\,\,\,s_{-3}
\end{split}
\end{align}
 \end{lem}

 \begin{lem}
In the case  $m_{-1}<0$ in the space of functions  \eqref{reshd}
that satisfy \eqref{usld} there exists a base indexed by  integer or
half-integer tableaux of type

\begin{align*}
&m_{-3}\,\,\,\,\,\,\,\\
&\,\,\,\,\,\,\,\,\,\,\,\,k_{-3}\,\,\,\,\,\,   \mathcal{D}   \,\,\,\,\,\,\,\,\,\,\,\,\,\,\,\\
&\,\,\,\,\,\,\,\,\,\,\,\,\,\,\,\,\,\,\,\,\,\,\,\,s_{-3}
\end{align*}
 where  $\mathcal{D}$  is a tableau for   $\mathfrak{o}_4$ with the highest weight
 $[m_{-2},m_{-1}]$.
 % Диаграмма определяется  показателями
 %$q_{-1}=p_{-3,-1},q_{-2}=p_{-3,-2},q=_{-2,-1}=p_{-3,-2,-1},q_{-1,1}=p_{-3,-1,1},q_{-2,1}=p_{-3,-2,1}$ или
 %  $q_{-2,0}=p_{-3,-2,0}$  и заданным выше числом  $k_1$.
 \end{lem}

\subsection{Considerations in the case of arbitrary $n$}

To consider the case of an arbitrary  $n$  we consider the sequence of extensions

$$
g_{n}\downarrow g_{n-1}\supset g_{n}\downarrow g_{n-1}\supset...\supset g_{3}\downarrow g_{2},
$$

 As it is done above,  one can show that all these extensions are isomorphic to the extension $\mathfrak{gl}_{3}\downarrow\mathfrak{gl}_{1}\supset\mathfrak{gl}_{2}\downarrow\mathfrak{gl}_{0}$.
The isomorhpism for futher extensions is analogous to those constructed in Section \ref{acde} for all series.

\subsection{Summary}

We come to the Theorem

 \begin{thm}
 \label{osnl1}
 Let us be given a  representation of  $g_n$ with the highest weight $[m_{-n},...,m_{-1}]$ in the cases  $B$, $C$, $D$ and
  $[m_{-n},...,m_{-1},0]$ in the case  $A$. Consider the  problem  $g_{n}\downarrow g_{n-1}$.
  Then in the space  of  $g_{n-1}$-vectors there exists a base given by tableau such that in the whole tableau (including  $\mathcal{D}$, except $s_{-2}$ in the of the series  $D$)  the betweness conditions hold

\begin{align}
\begin{split}
\label{diag}
&m_{-n}\,\,\,...\,\,\,\,m_{-3}\,\,\,\,\,\,\,\\
&\,\,\,\,\,\,\,\,\,\,\,\,k_{-n}\,\,\,...\,\,\,\,k_{-3}\,\,\,\,\,\,   \mathcal{D}   \,\,\,\,\,\,\,\,\,\,\,\,\,\,\,\\
&\,\,\,\,\,\,\,\,\,\,\,\,\,\,\,\,\,\,\,\,\,\,\,\,s_{-n}\,\,\,...\,\,\,\,s_{-3}
\end{split}
\end{align}
where  $\mathcal{D}$  is a tableau for   $g_2$ with the highest weight $[m_{-2},m_{-1}]$  for  $B$, $C$,   $D$  and   $[m_{-2},m_{-1},0]$  for  $A$.
All elements are simultaneously integer or half-integer  for seties  $B$, $D$
\end{thm}

\begin{cor}
\label{osncor}
The spaces of highest vectors with a fixed tableau   $\mathcal{D}$ are isomorphic for series $A$, $B$, $C$, $D$.
\end{cor}

Let us calculate a weight of a vector corresponding to a tableau.
The $(-k)$-th component of the weight is a sum of powers of
determinants that contain  $(-k)$ minus the sum of powers of
determinants that contain  $k$.

\begin{prop}
The lower row is a   $g_{n-1}$-weight of the corresponding $g_{n-1}$-highest vector.
\end{prop}

\proof
Consider the case $n=3$, the case of an arbitrary $n$ is considered analogously. The index $(-2)$  is contained only in those determinants that participate in construction of   $\mathcal{D}$.  Using the statement for  $g_2$-tableau we conclude that $s_{-2}$ is a    $(-2)$-component of the weight.

Now consider the  $(-3)$-th component of the weight. The sum of powers,  that contain  $(-3)$ equals to  $k_{-2}$. The sum of powers of determinants that contain   $(-3)$ and  that participate in the correspondence \eqref{mainsoot}, \eqref{mainsootd1}, \eqref{mainsootd2}  equals to
 $s_{-3}-k_{-2}$. Thus the sum of all powers equals to  $k_{-2}$.

\endproof

Let us give a formula for the eigenvalue of  $F_{-1,-1}$ or $F_{1,1}$.

 \begin{prop}\label{strokan}
The eigenvalue of  $F_{-1,-1}$  on the tableau \eqref{diag}
\begin{align*}
&-2\sum_{i=n}^{1}k_{-i}+\sum_{i=n}^1m_{-i}+\sum_{i=n}^2s_{-i} \text { in the case } \mathfrak{sp}_{2n},\mathfrak{o}_{2n},\\
&-2\sum_{i=n}^{1}k_{-i}+\sum_{i=n}^1m_{-i}+\sum_{i=n}^2s_{-i} +\sigma\text { in the case } \mathfrak{o}_{2n+1},\\
&-2\sum_{i=n}^{2}k_{-i}+\sum_{i=n}^1m_{-i}+\sum_{i=n}^2s_{-i} \text { in the case }\mathfrak{o}_{2n} \\
\end{align*}
In the case  $\mathfrak{o}_{2n}$ and
 $m_{-1}<0$  the eigenvalue of  $F_{1,1}$ on the tableau \eqref{diag}
\begin{align*}
&-2\sum_{i=n}^{2}k_{-i}+\sum_{i=n}^2m_{-i}-m_{-1}+\sum_{i=n}^2s_{-i} 
\end{align*}

\end{prop}

\proof
Consider first the cases  $A$, $B$, $C$ and  $D$ with  $m_{-1}\geq 0$.
Note that $\sum_{k=-n}^{3}(m_{-i}-k_{-i})$ is a sum of powers of determinants of order  less $n-2$ that contain  $1$,  and
$\sum_{k=-n}^{3}(k_{-i}-s_{-i})$ is a sum of powers of determinants of order  less  $n-1$ that contain   $-1$.  This follows from the fact that correspondences  \eqref{mainsoot}, \eqref{mainsootd1}, \eqref{mainsootd2}   the indices  $-1$ and  $1$  are preserved. And after this correspondence the sum of powers of determinanrts that contain   $1$  or   $-1$ can be calculated using the standart rulers of  calculation of the weigh of a  $\mathfrak{gl}_3$-tableau.

Now we have to calculate analogous sums for determinant of orders
$n-2$ and $n-1$. For series  $A$, $C$ these sums equal
$(m_{-2}-k_{-2})+(m_{-1}-k_{-1})$ and $(k_{-2}-s_{-2})$. For series
$B$ these sums equal   $(m_{-2}-k_{-2})+(m_{-1}-k_{-1})$ и
$(k_{-2}-s_{-2})+\sigma$.  This follows from the ruler of calculation of the weight of a $g_2$-tableau.

Now take the difference of these sums,
then one obtains the statement of the Proposition  for series $A$,
$C$. Thus in these cases we know the sum of powers of determinants
that contain $1$ and the sum of powers of determinants that contain
$-1$. Taking the difference of them we obtain the needed expression.

For series $D$ and $m_{-1}\geq 0$ the considerations are slightly
different: the formula for the difference of these sums of powers
for the determinant of orders $n-2$ and $n-1$  is given in
Proposition \ref{stroka041}.
% \ref{stroka042}.

The case $D$ with  $m_{-1}<0$ is considered analogously, but one must interchange $1$ and $-1$.

\endproof


\begin{thebibliography}{99}


\bibitem{zh2}D. P. Zhelobenko,  Compact Lie groups and their representations. – American Mathematical Soc., 1973. – Т. 40.

\bibitem{sh1} V. V. Shtepin,  Separation of multiple points of spectrum in the reduction
 $\mathfrak{sp}_{2n}\downarrow \mathfrak{sp}_{2n-2}$, Functional Analysis and Its Applications, 1986, 20:4, 336–338

\bibitem{sh2} V. V. Shtepin,  The intermediate orthogonal Lie algebra $\mathfrak{b}_{n-1/2}$ and its finite-dimensional representations
 Izvestiya: Mathematics, 1998, 62:3, 627–648

\bibitem{sh3}  V. V. Shtepin,   The intermediate Lie algebra  $\mathfrak{d}_{n-1/2}$, the weight scheme and finite-dimensional representations with highest weight, Izvestiya: Mathematics, 2004, 68:2, 375–404

\bibitem{Mol1} A. Molev, A basis for representations of symplectic Lie algebras, Comm. Math. Phys.
201, 1999, 591--618.

\bibitem{Mol2}A. Molev, Weight bases of Gelfand-Tsetlin type for representations of classical Lie algebras, J. Phys. A: Math. Gen. 33 (2000), 4143--4168.

\bibitem{Mol3}A. Molev,  A weight basis for representations of even orthogonal Lie algebras, Adv. Studies in Pure Math. 28 (2000), 223--242.

\bibitem{M}A. Molev. Yangians and classical Lie algebras, 2007, AMS, Mathematical Surveys and Monographs, vol. 143.

\bibitem{bb} G.E.  Biedenharn,  L.C. Baid, On the representations of semisimple Lie Groups II, J. Math. Phys., V. 4, N 12,
1963, 1449-1466.


\bibitem{a2}D.V. Artamonov,  The Gelfand-Tsetlin-Zhelobenko  base vectors for the series $B$,  arXiv:1607.08704v2

\end{thebibliography}
\end{document}